\documentclass[12pt]{amsart}
\usepackage{etex}
\usepackage[margin=1in]{geometry}

\usepackage{standalone}

\usepackage{tikz}
\usepackage{tikz-cd}

\usepackage{amssymb}
\usepackage{mathtools}
\usepackage{thmtools,thm-restate}

\usepackage{parskip}

\usepackage{etoolbox}
\AtEndEnvironment{thm}{\vspace{-0.5\baselineskip}}
\AtEndEnvironment{prop}{\vspace{-0.5\baselineskip}}
\AtEndEnvironment{lem}{\vspace{-0.5\baselineskip}}
\AtEndEnvironment{cor}{\vspace{-0.5\baselineskip}}
\AtEndEnvironment{rem}{\vspace{-0.5\baselineskip}}

\theoremstyle{plain}
\newtheorem{thm}{Theorem}[subsection]
\newtheorem{prop}[thm]{Proposition}
\newtheorem{cor}[thm]{Corollary}

\newtheorem{lem}[thm]{Lemma}

\theoremstyle{definition}
\newtheorem{defn}[thm]{Definition}
\newtheorem{ex}[thm]{Example}
\newtheorem{rem}[thm]{Remark}

\newtheorem*{ack}{Acknowledgements}

\usepackage[all]{xy}

\usepackage{graphicx}
\graphicspath{{Images/}}
\usepackage{subcaption}

\usepackage{todonotes}

\usepackage[shortlabels]{enumitem}
\setlist[enumerate]{label=(\arabic*)}

\usepackage[british]{isodate}

\usepackage{bm}

\usepackage{hyperref}

\newcommand{\Z}{\mathbb{Z}}
\newcommand{\R}{\mathbb{R}}
\DeclareMathOperator{\Aut}{Aut}
\DeclareMathOperator{\Out}{Out}
\DeclareMathOperator{\Inn}{Inn}
\DeclareMathOperator{\Isom}{Isom}
\DeclareMathOperator{\Fix}{Fix}

\DeclareMathOperator{\GL}{GL}
\DeclareMathOperator{\Inc}{Inc}
\DeclareMathOperator{\PGL}{PGL}

\DeclareMathOperator{\Ad}{Ad}
\DeclareMathOperator{\MC}{Mc}
\DeclareMathOperator{\Hom}{Hom}
\DeclareMathOperator{\MCG}{MCG}

\newcommand{\fix}{\mbox{\rm Fix } \phi}
\newcommand{\fp}[1]{\mbox{\rm Fix } #1}

\newcommand{\rank}{\mbox{\rm Rank } }

\let\phi\varphi

\title[Automorphisms of free-by-cyclic groups]{Free-by-cyclic groups, automorphisms and actions on nearly canonical trees}
\author{Naomi Andrew and Armando Martino}
\address{Mathematical Sciences, Building 54, University of Southampton, Southampton, SO17 1BJ}
\email{N.G.Andrew@soton.ac.uk}
\email{A.Martino@soton.ac.uk}

\date{}

\begin{document}


\begin{abstract}
We study the automorphism groups of free-by-cyclic groups and show these are finitely generated in the following cases: (i) when defining automorphism has linear growth and (ii) when the rank of the underlying free group has rank at most 3. 

The techniques we use are actions on trees, including the trees of cylinders due to Guirardel and Levitt
, the relative hyperbolicity of free-by-cyclic groups (due to Gautero and Lustig, Ghosh, and Dahmani and Li) 
 and the filtration of the automorphisms of a group preserving a tree, following Bass and Jiang, and Levitt.

Our general strategy is to produce an invariant tree for the group and study that, usually reducing the initial problem to some sort of McCool problem (the study of an automorphism group fixing some collection of conjugacy classes of subgroups) for a group of lower complexity. The obstruction to pushing these techniques further, inductively, is in finding a suitable invariant tree and in showing that the relevant McCool groups are finitely generated.
\end{abstract}

\maketitle

\section{Introduction}

\subsection{Free-by-cyclic groups}

Given a finite rank free group $F_n$ and an automorphism $\varphi \in \Aut(F_n)$, we can define a free-by-cyclic group $G = F_n \rtimes_\varphi \langle t \rangle = \langle x_1, \dots x_n, t | t^{-1}x_it=x_i\varphi \rangle$ (so conjugating by the stable letter $t$ acts on $F_n$ as the automorphism $\varphi$). The properties of this free-by-cyclic group depend only on the automorphism $\varphi$, and in fact only on the conjugacy class of its image in the outer automorphism group, $\Phi$ \cite[Lemma 2.1]{Rank2Case}. 

Various properties of $G$ follow from $\varphi$ and indeed from $\Phi$: for example, $G$ is hyperbolic if and only if $\varphi$ is atoroidal (no power of $\varphi$ fixes the conjugacy class of an element in $F_n$) \cite{BrinkmannHyperbolicity}, and is relatively hyperbolic if and only if the length of some word in $F_n$ grows exponentially under iteration of $\varphi$ \cite{GauteroLustigRelHyp,GhoshRelHyp,DahmaniLiRelHyp}. Both of these properties are properties of the outer class as a whole.

In this paper we study the actions of free-by-cyclic groups on trees, and through this their automorphisms. Even in rank 1 (the two cyclic-by-cyclic groups) it is hard to say anything very general about their automorphisms: for $\Z^2$, the outer automorphism group is $\GL(2,\Z)$, whereas for the fundamental group of the Klein bottle it has only four elements.

There are groups which can be expressed as free-by-cyclic groups with more that one possibility for the rank of $F_n$. However,  there are some things which these presentations will have in common: for example, the growth rate of the outer automorphism $\Phi$ will be the same~\cite{MacuraQIinvariance}. An automorphism is polynomially growing (with degree $d$) if, as it is iterated, the conjugacy length of a word is bounded by a polynomial (of degree $d$), and (by a Theorem of \cite{BestvinaHandel1992} -- see Subsection~\ref{growth rate}) exponentially growing otherwise. We split our investigation by growth rate. 

Levitt's work on Generalised Baumslag-Solitar groups \cite{LevittGBS} includes (after checking some hypotheses) that if the defining (outer) automorphism is finite order (in which case the free-by-cyclic group $G$ is virtually $F_n \rtimes \Z$) then $\Out(G)$ is VF, and so in particular finitely generated. 

We extend finite generation to all cases where the defining outer automorphism has linear growth:

\begin{restatable}{thm}{linear}
\label{thm:linear}
	Suppose $G \cong F_n \rtimes_\varphi \Z$, and $\varphi$ is linearly growing. Then $\Out(G)$ is finitely generated.
\end{restatable}

Also \cite{Rank2Case}, studies the case when the underlying free group has rank 2. There $\Out(G)$ is calculated up to finite index for all defining automorphisms, and this classification shows that $\Out(G)$ is finitely generated.

We extend the finite generation result to all cases where the underlying free group has rank 3 (in which case the growth is either at most quadratic or exponential):

\begin{restatable}{thm}{rankthree}
\label{thm:rank3}
	Suppose $G \cong F_3 \rtimes \Z$. Then $\Out(G)$ is finitely generated.
\end{restatable}

We understand the automorphism groups through studying certain actions of $G$ on trees. Since they are defined as HNN extensions, all free-by-cyclic groups have a translation action on the real line. But they also admit actions on more complicated trees. These actions are equivalent to alternative presentations which can provide more information about the group. To understand the automorphisms, we use particular trees which are in some sense invariant under all -- or sometimes only most -- automorphisms.

The details are different in the exponentially growing and polynomially growing cases. With exponential growth, $G$ is one-ended relatively hyperbolic, and so it has a canonical JSJ decomposition by~\cite{GuirardelLevitt2015}. These decompositions are particularly useful and well understood, and there is a description of the outer automorphism group arising from them. We describe the canonical tree and for the low rank cases carry out the calculations needed for the automorphism group in Section 4.

Using Guirardel and Levitt's tree of cylinders construction~\cite{GuirardelLevittTreesOfCylinders}, we construct canonical trees when the defining automorphism is unipotent polynomially growing (UPG) and either linear or, in low rank, quadratic. These trees arise from fixed points on the boundary of Culler-Vogtmann outer space for the defining (outer) automorphism and restricting the action to $F_n$ will give an action in the same deformation space as such a tree. Every polynomially growing  automorphism has a power which is UPG (in fact, the power can be taken to depend only on the rank of the free group -- see Definiton~\ref{defn:UPG}). This implies the existence of a normal finite index subgroup which is again free-by-cyclic, this time with a UPG defining automorphism.

Understanding the automorphisms of a finite index subgroup does not necessarily provide insight into those of the larger group: the fundamental group of a Klein bottle (with only four outer automorphisms) contains $\Z^2$ as an index 2 subgroup. A key part of our proof is that we can use the existence of a canonical splitting of a normal finite index subgroup to find a splitting of the larger group which is ``nearly canonical'' -- invariant under at least a finite index subgroup of automorphisms, see Definition~\ref{defnearlycanonical}.

The result is:

\begin{restatable*}{prop}{NearlyCanonical}
	\label{prop:no_really_theres_an_action}
	Let $G$ be a finitely generated group, $G_0$ a normal finite index subgroup of $G$, and suppose that $T$ is a canonical $G_0$-tree. Then \begin{enumerate}[(i)]
		\item $G$ acts on $T$, and this action restricts to the canonical $G_0$-action.
		\item With this action, $T$ is nearly canonical as a $G$-tree.
	\end{enumerate}
\end{restatable*}

By this result, we have an action of $G$ on a tree, and we can consider the outer automorphisms preserving this action, which is finite index in the full outer automorphism group. Understanding this group depends on understanding the vertex and edge groups, their automorphisms, and how those automorphisms interact. In particular, we need to calculate ``McCool groups'', for vertex groups with respect to adjacent edge groups: the outer automorphisms having representatives that restrict to the identity on each of a family of subgroups.

As part of our proof, we carry this calculation out for free-by-cyclic groups defined by a periodic automorphism, with respect to  a limited class of subgroups, and when the underlying free group has rank 2.

We note that there appear to be two main obstacles to extending this result further: constructing actions on trees which are (nearly) canonical, and understanding the McCool groups arising from these trees. In the exponential case, the canonical trees exist and the obstruction is only the McCool groups, which are generally required to be with respect to fairly complex subgroups. In the polynomially growing case(s) passing to a (UPG) power should lead to actions arising from limit points of $CV_n$, and with quadratic growth these are even unique -- see \cite{Lustig2017}. But it is not obvious that the deformation spaces these define are canonical. If canonical trees can be found, the McCool groups are likely to be needed relative only to infinite cyclic subgroups, which may be more manageable.

\begin{ack}
The first author was supported by an EPSRC studentship, and the second author by Leverhulme Trust Grant RPG-2018-058. We would also like to thank Gilbert Levitt and Ashot Minasyan for many helpful comments. 
\end{ack}

\section{Background}

\subsection{Notation, Actions on trees and Bass-Serre Theory}
\label{notation}
We record here some notation for actions on trees and various subgroups of (outer) automorphisms used throughout the paper.

We recall enough of Bass-Serre theory to set notation; see \cite{Serre2003} amongst others for a fuller exposition.
Following Serre, the edges of a graph come in pairs denoted $e$ and $\overline{e}$, and $\iota(e)$ and $\tau(e)$ denote the initial and terminal vertices. An orientation, $\mathcal{O}$, is a choice of one edge from each pair $\{e,\overline{e}\}$.

Let a group $G$ act on a tree $T$. We let $G_v$ and $G_e$ denote the stabiliser of a vertex $v$ or edge $e$ respectively; from the perspective of graphs of groups we use them for vertex and edge groups, and use $\alpha_e$ to denote the monomorphism $G_e \to G_\iota(e)$. Often we simply identify $G_e$ with its image $\alpha_e(G_e)$ in $G_{\iota(e)}$. An action on a tree is called \textit{minimal} if it does not admit a $G$-invariant subtree; most of our actions will be assumed to be minimal. An action on a tree is \emph{irreducible} if it does not fix a point, line, or end of the tree; to guarantee this it is sufficient that the action has two hyperbolic axes whose intersection is at most finite length. 

We use $N_G(H)$, $C_G(H)$ and $Z(H)$ for the normaliser of $H$ in $G$, the centraliser of $H$ in $G$, and the centre of $H$. To save space and subscripts in the context of an action on a tree, we let $N_e = N_{G_{\iota(e)}}(G_e)$ and $C_e = C_{G_{\iota(e)}}(G_e)$.

As usual, $\Aut(G)$ denotes the automorphisms of $G$, and $\Out(G)=\Aut(G)/\Inn(G)$ the outer automorphisms. We use lower case greek letters ($\varphi$) for automorphisms, and upper case ($\Phi$) for outer automorphisms. If the image of $\varphi$ in $\Out(G)$ is $\Phi$, we say $\varphi$ \emph{represents} $\Phi$, or $\varphi \in \Phi$ (viewing $\Phi$ as a coset of $\Inn(G)$).

Given an automorphism $\varphi$ of $G$, we can define the cyclic extension of $G$ by $\varphi$ as \[G\rtimes_\varphi \Z = \langle X,t : R, t^{-1}xt=x\varphi \rangle\] (taking $\langle X : R \rangle$ to be a presentation of $G$). Automorphisms representing the same outer automorphism define isomorphic extensions, as can be seen by introducing a new generator $t'=tg$. For this reason, we will sometimes use the notation $G \rtimes_\Phi \Z$ to refer to the isomorphism class of cyclic extensions defined by any automorphism representing $\Phi$.

For $g \in G$, we write $\Ad(g)$ for the inner automorphism of $G$ induced by $g$. 

If $G$ normalises $H$, let $\Ad(G,H)$ denote the automorphisms of $H$ induced by conjugating by elements of $G$. If $H$ is clear, it may be omitted; in particular (and assuming an action on a tree) $\Ad(N_e)$ always means the automorphisms of $G_e$ induced by conjugating by $N_e$, its normaliser in $G_{\iota(e)}$. Notice that since $N_e$ contains $G_e$, the subgroup $\Ad(N_e)$ descends to a subgroup of $\Out(G_e)$.

We identify certain ``relative'' subgroups of $\Out(G)$:
\begin{defn}
\label{def:mccool_etc}
Given a family of subgroups $\{G_i\}$ of $G$, we define \begin{itemize}
\item the subgroup $\Out(G;\{G_i\})$ to be those outer automorphisms of $G$ where for each subgroup $G_i$ there is a representative that restricts to an automorphism of $G_i$.
\item the subgroup $\MC(G;\{G_i\})$ to be those outer automorphisms of $G$ where for each subgroup $G_i$ there is a representative that restricts to the identity on $G_i$.
\end{itemize}\end{defn}
Note that these are subgroups of \emph{outer} automorphisms; any given representative will not usually have the correct restriction for every subgroup $G_i$. 

Throughout the paper we consider actions of $F_n \rtimes_\varphi \Z$ on trees. Dahmani (in Section 2.2 of \cite{Dahmani2016}) gives some useful results about such an action. The following lemma is specialised to free-by-cyclic groups; Dahmani gives it more generally for semidirect products with $\Z$ (suspensions, in the terminology of that paper) of any finitely generated group.

\begin{lem}
\label{lem:stabs_are_free_by_cyclic}
	Suppose $G \cong F_n \rtimes \langle t \rangle$ acts minimally and irreducibly on a tree. Then \begin{enumerate}
		\item $F_n$ acts on the same tree with finite quotient graph
		\item The stabilisers in any action of $G$ on a tree are again free-by-cyclic; the free part is the $F_n$-stabiliser, and the generator of the cyclic factor has the form $t^kw$.
		\item In particular, all edge stabilisers are at least infinite cyclic, and $G$ is one ended.
		\item If all incident edges at some vertex are cyclically stabilised, its stabiliser cannot be finitely generated and infinitely ended.
	\end{enumerate}
\end{lem}

The hypotheses given here differ slightly from Dahmani's: we demand an irreducible action while Dahmani uses ``reduced''. In fact a sufficient condition is that $F_n$ acts non-trivially, which is guaranteed by either of these conditions. 

Note that the last point is not immediately obvious: the free part at a vertex could be infinitely generated, and free by cyclic groups of this form can be infinitely ended, and even free. It implies that in any splitting of this kind there cannot be a ``quadratically hanging'' vertex group, for these have exactly the combination of properties ruled out here (see Section~\ref{sec:exp_growth}).

The crucial observation is that since $F_n$ is a normal subgroup, it also acts minimally on the whole tree. Then finite generation ensures the quotient under this action is finite. To recover the free-by-cyclic structure on the stabilisers, consider the action of $\langle t \rangle$ on the quotient graph, and lift the stabilisers back to the whole group. This kind of argument enables us to analyse the splitting of $G$ by considering the induced splitting of $F_n$.

To see the last point, in this case, note that this could only occur if the free part of the relevant vertex stabiliser was not finitely generated. Contract all other edges and consider the induced free splitting of $F_n$: this expresses $F_n$ as a free product where one free factor is not finitely generated, which is impossible. 

(More generally, the free part of a vertex may only be infinitely generated if the same is true of at least one incident edge group; control over the edge groups provides some control over the vertex groups.)

\subsection{Length Functions and Twisting Actions by Automorphisms}

Since we will usually be working with simplicial metric trees, an action of $G$ on a tree $T$ will be equivalent to a map $G \to \Isom(T)$. 

Any action of $G$ on a tree $T$ defines a \emph{translation length function} on $G$, by considering the minimum displacement of points in the tree for each element. That is, given an isometric action of $G$ on $T$, we can define the function, $l_T: G \to \R$ by $l_T(g) = \min_{x \in T} d_T(x, xg)$ (and this minimum is always realised). Note that $l_T$ is constant on conjugacy classes.

We recall a well-known Theorem of Culler and Morgan: 

\begin{thm}[{\cite[Theorem 3.7]{CullerMorgan1987Rtrees}}] 
	\label{CMlengthfn}
	Let $G$ be a finitely generated group and let $T_1, T_2$ be two $\R$-trees equipped with isometric $G$ actions which are minimal and irreducible. Then $l_{T_1}=l_{T_2}$ if and only if $T_1$ and $T_2$ are equivariantly isometric. Moreover, such an equivariant isometry is unique if it exists.
\end{thm}

\begin{rem}
	This result says that, in many cases, the translation length function determines the action. 
\end{rem} 

The action of $G$ on $T$ defines a deformation space, by considering all simplicial actions of $G$ on a tree with the same elliptic subgroups (this is an equivalence relation on $G$-trees); the elliptic subgroups are those subgroups of $G$ which fix a point in the tree. Note that there can still be vertices with stabilisers that are not conjugate to a stabiliser in the original action. (For example, consider representing a free product of three groups as a graph of groups where the underlying graph is a line versus a tripod: these are in the same deformation space, despite the extra trivially stabilised vertex.) Trees in the same deformation space \emph{dominate} each other; that is, there are equivariant maps between them.

\begin{defn}
	\label{twistedactions}
	Given an isometric action of a group $G$ on a tree, $T$, a new `twisted' action of $G$ on $T$ can be defined by pre-composing with any automorphism of $G$. That is, if $\varphi \in \Aut(G)$, then $x \cdot_\varphi g = x \cdot (g\varphi)$ 
	
	In terms of length functions, this means that $l_{ \phi T} (g) = l_T(g \phi)$. (Here $\varphi T$ is the ``twisted tree'', isometric to $T$ but with the new action defined above.) 
	
	Given a deformation space of trees, this defines an action of $\Aut(G)$ on that space. 		
\end{defn}

\begin{rem}
	Note that there is a switch from left to right; if the automorphisms of $G$ act on elements on the right then the action on trees by pre-composing is on the left and vice versa.
\end{rem}

In most cases this changes the length function; we let \[\Aut^T(G)=\{ \phi \in \Aut(G) \ : \ l_{\phi T} = l_T \}\] denote the subgroup of $\Aut(G)$ which leaves it unchanged. Notice that this is true of all inner automorphisms, so these are a subgroup of $\Aut^T(G)$. By Theorem~\ref{CMlengthfn} such an automorphism induces an equivariant isometry of $T$, and assuming the action is minimal and does not fix an end this is unique and extends to an action of $\Aut^T(G)$. This action is compatible with the original action in at least two senses: $G \rtimes \Aut^T(G)$ (with the usual action of $\Aut^T(G)$ on $G$) acts on $T$, restricting to the original action of $G$ and the induced action of $\Aut^T(G)$, and the action of $G$ on $T$ factors through the map sending each element to the inner automorphism it induces.

This is asserting the existence of a commuting diagram (see \cite{Andrew2021} for how to use Theorem~\ref{CMlengthfn} to produce this diagram):
\[
\begin{tikzcd}
	\Aut^T(G) \arrow[dotted]{r} & \Isom(T) \\
	G \arrow{u}{g \mapsto \Ad(g)} \arrow{ur}{\cdot}
\end{tikzcd}
\]

Recall (from Section~\ref{notation}) that $\Ad(g)$ is the inner automorphism induced by $g$. 

In fact, such a diagram is also sufficient to recover the definition in terms of length functions
since for any $\phi \in \Aut^T(G)$, 
$$
l_T(g) = l_T(\Ad(g)) = l_T(\Ad(g)^\phi) = l_T(\Ad({g \phi})) = l_T(g \phi),
$$
(where we are moving between the $G$ action and the $\Aut^T(G)$ action using the commutative diagram). 

We will consider $\Out^T(G)=\Aut^T(G)/\Inn(G)$, the subgroup of outer automorphisms which preserves the length function. By the correspondence theorem, many properties of $\Aut^T(G)$, such as finite index or normality, are inherited by $\Out^T(G)$.

\subsection{Trees of cylinders}

\label{treesofcylinders}
Guirardel and Levitt in \cite{GuirardelLevittTreesOfCylinders} define a tree of cylinders for a deformation space. The input is any tree in the deformation space, and an equivalence relation on the edges; the output is a tree where the induced splitting is preserved by all (outer) automorphisms which preserve the deformation space. They are our main tool for producing trees which allow us to analyse outer automorphisms by considering trees.


We start the construction by defining a family $\mathcal{E}$ of subgroups of $G$. It should be closed under conjugation, but not under taking subgroups. We then define an \emph{admissible equivalence relation} on $\mathcal{E}$~\cite[Definition 3.1]{GuirardelLevittTreesOfCylinders}. This must satisfy \begin{enumerate}
\item if $A \sim B$ then $A^g \sim B^g$ for all $g \in G$
\item if $A \leq B$ then $A \sim B$
\item Suppose $G$ acts on a tree with stabilisers in $\mathcal{E}$. If $A\sim B$, $v \in \Fix(A)$ and $w\in \Fix(B)$, then the stabiliser of any edge lying in $[v,w]$ is equivalent to $A$ (and $B$)
\end{enumerate}

To show (3) it is sufficient to show that $\langle A, B \rangle$ is elliptic~\cite[Lemma 3.2]{GuirardelLevittTreesOfCylinders}

Now suppose that $G$ acts on $T$ with edge stabilisers in $\mathcal{E}$. Define an equivalence relation on the edges of $T$ by saying $e \sim e'$ if $G_e \sim G_{e'}$. A \emph{cylinder} consists of an equivalence class of edges; the conditions on an admissible equivalence relation ensure that cylinders are connected, and that two cylinders may intersect in at most one vertex.

To construct the tree of cylinders $T_c$, replace each cylinder with the cone on its boundary~\cite[Definition 4.3]{GuirardelLevittTreesOfCylinders}. That is, there is a vertex $Y$ for every cylinder, together with surviving vertices $x$ lying on the boundary of two (or more) cylinders. Edges show inclusion of a boundary vertex $x$ into a cylinder $Y$. The stabilisers of boundary vertices are unchanged; the stabiliser of a cylinder vertex is the (setwise) stabiliser of the cylinder. Edge stabilisers are the intersection of the relevant vertex stabilisers.

The tree of cylinders $T_c$ depends only on the deformation space of $T$, in the sense that given two minimal, non-trivial trees $T,T'$ in the same deformation space, there is a canonical equivariant isomorphism between $T_c$ and $T_c'$~~\cite[Corollary 4.10]{GuirardelLevittTreesOfCylinders}. In particular this means that this tree of cylinders is fixed by any automorphism which preserves the deformation space, and so can be used to study these (outer) automorphisms.

It is always true that $T$ dominates $T_c$, but cylinder stabilisers may not be elliptic in $T$. The deformation space of the tree of cylinders depends on the size of the cylinders: if all cylinders are bounded, or equivalently contain no hyperbolic axis, then the cylinder stabilisers are elliptic in $T$ and so $T_c$ lies in the same deformation space, and conversely~\cite[Proposition 5.2]{GuirardelLevittTreesOfCylinders}.

Edge stabilisers may not be in $\mathcal{E}$; in this case the \emph{collapsed tree of cylinders} $T_c^\ast$ is defined by collapsing all edges of $T_c$ with stabilisers not in $\mathcal{E}$~\cite[Definition 4.5]{GuirardelLevittTreesOfCylinders}. Assuming that $\mathcal{E}$ is sandwich closed (if $A \leq B \leq C$ are subgroups of $G$, and $A$ and $C$ are in $\mathcal{E}$, then so is $B$), the construction is stable in the sense that $(T_c^\ast)_c^\ast=T_c^\ast$~\cite[Corollary 5.8]{GuirardelLevittTreesOfCylinders}. If $T$ and $T'$ are in the same deformation space then there is a unique equivariant isometry between $T_c^\ast$ and $(T')_c^\ast$~\cite[Corollary 5.6]{GuirardelLevittTreesOfCylinders} and again this action is canonical.

In general, there may be more restrictions put on the trees: sometimes we require that a certain collection of subgroups is elliptic. In this case the deformation space and tree of cylinders is canonical relative to the automorphisms which preserve this collection.

\subsection{Automorphisms of free groups}

We recall some of the results we will use about automorphisms of free groups. Here $\Aut(F_n)$ denotes the automorphism group of the free group of rank $n$ and $\Out(F_n) = \Aut(F_n)/\Inn(F_n)$, the group of outer automorphisms, which is the quotient by the inner automorphisms. Thus an outer automorphism is a coset of inner automorphisms, and there is an equivalence relation on this set of automorphisms called {\em isogredience}. Formally, 

\begin{defn}
	Two automorphisms $\phi, \psi \in \Aut(F_n)$ are said to be\textit{ isogredient }if they are conjugate by an inner automorphism. This is an equivalence relation when restricted to any coset of $\Inn(F_n)$, that is an element of $\Out(F_n)$. 
\end{defn}

\begin{thm}[Bestvina-Handel Theorem, \cite{BestvinaHandel1992}]
	\label{Bestvina-Handel}
	Let $\Phi \in \Out(F_n)$. Then,
	\[
	\sum \max\{ \rank(\fix) - 1, 0 \} \leq n-1,
	\]
	where the sum is taken over representatives, $\phi$, of isogredience classes in $\Phi$. 
\end{thm}

Note that isogredient automorphisms have conjugate fixed subgroups, so the ranks of the fixed subgroups do not depend on the representatives chosen.

\subsection*{Growth Rate}
\label{growth rate}

If we fix a basis, $B$, of $F_n$ then we set $\|g\|_B$ to be the length of the shortest conjugacy class of $g$ with respect to $B$, for any $g \in F_n$. We simply write this as $\|g\|$ if $B$ is understood. 

Given a $\Phi \in \Out(F_n)$ it is then clear that there exists a $\lambda$ such that, 
\[
\frac{\|\Phi^k(g)\|}{\|g\|} \leq \lambda^k,
\]
as we can simply take $\lambda$ to be the maximum conjugacy length of the image of any element of $B$. (Also note that since these are conjugacy lengths, we can apply any automorphism in the same outer automorphism class and get the same result. Thus we are effectively applying an outer automorphism.)

One of the results of \cite{BestvinaHandel1992} is that the growth of elements in this sense is either exponential or polynomial. That is, for any $g \in F_n$, we either get that, for some $\mu < \lambda$, 
\[
\mu^k \leq \frac{\|\Phi^k(g)\|}{\|g\|} \leq \lambda^k,
\]
 or there exist constants $0 < A < B$ such that 
\[
Ak^d \leq \frac{\|\Phi^k(g)\|}{\|g\|} \leq Bk^d,
\]
where $d \in \{ 0,1, \ldots, n-1  \}$.

See~\cite[Theorem 6.2]{LevittInequalities} for a precise description of the growth types of elements of $F_n$.

Accordingly, we say that 

\begin{defn}
	$\Phi \in \Out(F_n)$ has exponential growth if there is some element $g$ whose conjugacy length grows exponentially. And we say that $\Phi$ has polynomial growth of degree $d$ if the conjugacy length of every element grows polynomially and $d$ is the maximum degree of these polynomials. 
\end{defn}

Note that in our usage ``polynomial growth of degree $d$'' implies that $d$ is the smallest degree bounding the growth of every element: so for example for an automorphism of quadratic growth there will be an element whose conjugacy length grows quadratically.

We note that the property of having exponential or polynomial growth (and the degree of polynomial growth) are independent of the basis, $B$. Also, the growth type (although not the exponential growth rate) of an automorphism is the same as that of its powers. (This includes negative powers, though this is a harder fact to verify). 

\medskip

\subsection*{UPG Automorphisms}

We shall look at (outer) automorphisms of polynomial growth and consider a subclass of these, called the \emph{UPG automorphisms}.

\begin{defn} (see \cite{Bestvina2000}, Corollary 5.7.6)
	\label{defn:UPG}
	We say that $\Phi \in \Out(F_n)$ is \emph{Unipotent Polynomially Growing}, or UPG, if it has polynomial growth and it
	has unipotent image in $\GL_n(\Z)$. This is guaranteed if the automorphism induces the trivial map on the homology group of $F_n$ with $\Z_3$ coefficients. 
	
	Hence, any polynomially growing automorphism has a power which is UPG.  Moreover, this power can be taken to be uniform (given $n$). 
\end{defn}

In the subsequent arguments we will have need to refer to a particular type of free group automorphism called a \textit{Dehn Twist}. These were defined in terms of certain maps via a graph of groups - multi-twists of a graph of groups with maximal cyclic edge groups. Namely, one takes a splitting of the free group with infinite cyclic edge groups and looks at a map defined by ``twisting" along the edges. For our purposes, we shall define them as linear growth UPG automorphisms. 

\begin{defn}
	\label{DehnandUPG}
	Let $\Phi \in \Out(F_n)$. Then $\Phi$ is called a Dehn Twist automorphism of $F_n$ if it is UPG and has linear growth. 
\end{defn}

However, this definition is equivalent to that of a multi-twist.  

\begin{thm}(see \cite{Cohen1999}, \cite{Krstic2001}, \cite{BestvinaHandel1992} and \cite{Bestvina2000})
	Multi-twist automorphisms of free groups, as defined in \cite{Cohen1995} and \cite{Cohen1999},  are precisely the linear growth UPG automorphisms. Thus, these are both Dehn Twists. 
\label{dehnupg}
\end{thm}

\begin{rem}
	As commented in \cite{Krstic2001}, Theorem~\ref{dehnupg} is not proved explicitly in the papers cited, but is well known to experts. The idea is that a UPG automorphism has a `layered' improved relative train track representative by \cite{Bestvina2000}. The fact that it has linear growth will imply that there are no attracting fixed points on the boundary, and from there is it relatively straightforward to produce a graph-of-groups description in terms of the `twistors' of \cite{Cohen1999}. The arguments in \cite{DehnTwists} show how to go from the relative train track map to the graph of groups description explicitly. 
\end{rem}

There is also another characterisation of Dehn Twists, as given by \cite{DehnTwists}. 

\begin{thm}[Theorem 4.2 of \cite{DehnTwists} and Corollary 7.7 of \cite{Cohen1999}]
	\label{definetwists}
	Let $\Phi \in \Out(F_n)$ be an outer automorphism of the free group of rank n. Then $\Phi$ is a Dehn Twist if (and only if)
	$$
		\sum \max\{ \rank(\fix) - 1, 0 \} = n-1,
	$$	
		where the sum is taken over representatives, $\phi$, of isogredience classes in $\Phi$. That is, the inequality in Theorem~\ref{Bestvina-Handel} is an equality. 	
\end{thm}


A crucial Theorem about Dehn Twists is the Parabolic Orbits Theorem, which requires a little notation to set up. The context is Culler-Vogtmann space, $CV_n$, which is the space of free, simplicial actions of $F_n$ on metric trees. In this formulation, two points -- actions on trees -- are said to be equivalent if there is an equivariant homothety between them. There is a compactification of this space, $\overline{CV_n}$, which turns out to be the space of \emph{very small actions} of $F_n$ on $\R$-trees. 
The precise definition is not necessary here, but it is worth noting that the compactification includes points which are actions on trees that are \emph{not} simplicial $\R$-trees. 

There is a natural action of $\Out(F_n)$ on $CV_n$ and $\overline{CV_n}$, as in Definition~\ref{twistedactions}, obtained by pre-composing the action by automorphisms. 

\begin{thm}[Parabolic Orbits Theorem -- see \cite{Cohen1995} and \cite{Cohen1999}]
\label{thm:parabolic_orbits}
	Let $\Phi \in \Out(F_n)$ be a Dehn Twist. Then for any $X \in CV_n$, $\lim_{k \to \infty} \Phi^k(X)=T \in \overline{CV_n} $ exists, is a simplicial tree and lies in the same simplex -- any two such limit trees are equivariantly homeomorphic -- independently of $X$. Moreover, $T$ is a simplicial $F_n$-tree with the following properties. 
	
	\begin{enumerate}[(i)]
		\item Edge stabilisers are maximal infinite cyclic 
		\item Vertex stabilisers are precisely the subgroups $\fix$, where $\phi \in \Phi$ has a fixed subgroup of rank at least 2. 
	\end{enumerate}
\end{thm}


Since inner automorphisms only fix infinite cyclic groups, and as any vertex stabiliser, $H$, has rank at least 2, then the corresponding automorphism $\phi \in \Phi$, such that \hbox{$H=\fix$,} is uniquely defined. 

Note that if we take two vertices of $T$ in the same orbit, then their stabilisers are conjugate, and the corresponding automorphisms are isogredient. (In general, having conjugate fixed subgroups is not enough to imply isogredience, but it is when the fixed subgroup has rank at least 2). Conversely, if two vertices are in different orbits then the corresponding automorphisms are not isogredient, since edge stabilisers are cyclic.

Given a Dehn Twist, $\Phi$ and a $\phi \in \Phi$, one can construct the free-by-cyclic group, \hbox{$G=F_n \rtimes_{\phi} \langle s \rangle$} (one can do this for any free-by-cyclic group, and the group does not depend on the choice of $\phi$). Using the parabolic orbits Theorem, one gets that $G$ acts on $T$, with the following properties, 

\begin{enumerate}[(i)]
	\item The induced action of $s$ on the quotient of $T$ by $F_n$ is trivial, 
	\item The $G$ edge stabilisers are maximal $\Z^2$, 
	\item The $G$ vertex stabilisers are $F_k \times \Z$, for $k \geq 2$, 
	\item The element $sg$ fixes a vertex of $T$ if and only if $\phi \Ad(g)$ has a fixed subgroup of rank at least 2 (equivalently, if $sg$ has non-abelian centraliser). 
\end{enumerate}

\section{Extending actions to the automorphism group}

\subsection{Canonical actions and nearly canonical actions}
Since outer automorphisms of free groups often have a power that is better understood (for us, usually a UPG power of a polynomially growing outer automorphism) it can be easier to work with the free-by-cyclic group defined by this power, which is a finite index subgroup of $G$. However, this means understanding how the automorphisms of a group and a finite index subgroup relate. In general this is hard: recall that the Klein bottle group has a finite outer automorphism group, but contains $\Z^2$ as a finite index subgroup.

Recall that, by Theorem~\ref{CMlengthfn}, the action of $G$ on a tree, $T$, is encoded by its translation length function, $l_T$. Our strategy is to show that $\Aut^T(G)$, the subgroup of automorphisms preserving the tree (or, equivalently, length function) is finitely generated. Thus we need to find a tree $T$, such that $\Aut^T(G)$ is either equal to $\Aut(G)$ or is a finite index subgroup of it. 

However, the proof of one of our key Lemmas (Lemma~\ref{prop:no_really_theres_an_action}) requires us to work with the actions directly rather than via length functions. Therefore we make the following definitions:  
\begin{defn}
An action of a group $G$ on a tree, $T$, is called \emph{canonical} if there exists a commuting diagram: 
$$
\begin{tikzcd}
	\Aut(G) \arrow{r} & \Isom(T) \\
	G \arrow{u}{g \mapsto \Ad(g)} \arrow{ur}{\cdot}.
\end{tikzcd}
$$

In the case where $G$ is finitely generated and the $G$-action is minimal and does not preserve an end, this is equivalent -- by Theorem~\ref{CMlengthfn} --  to translation length function being preserved by all of $\Aut(G)$, that is $\Aut^T(G)=\Aut(G)$ (and $\Out^T(G)=\Out(G)$).
\end{defn}

\begin{defn}
\label{defnearlycanonical}
We say that an action of $G$ on a tree, $T$, is called \emph{nearly canonical} if there is a finite index subgroup, $\Inn(G) \leq A \leq \Aut(G)$ such that the following diagram commutes: 
$$
\begin{tikzcd}
	A \arrow{r} & \Isom(T) \\
	G \arrow{u}{g \mapsto \Ad(g)} \arrow{ur}{\cdot}.
\end{tikzcd}
$$
In the case where $G$ is finitely generated and the $G$-action is minimal and does not preserve an end, this is equivalent to 
the translation length function being preserved by a finite index subgroup of automorphisms; that is $\Aut^T(G)$ is a finite index subgroup of $\Aut(G)$. 
\end{defn}

\begin{rem}
	We are not aware of this last definition in the literature, but it is clearly useful. Also, we have avoided calling this \textit{virtually canonical} since it would raise the confusion between what we mean, and canonical for a finite index subgroup of $G$. (See Proposition~\ref{prop:no_really_theres_an_action}). 
\end{rem}

We are able to extend a canonical action of a normal finite index subgroup to a nearly canonical action of the whole group, as shown in the following proposition. 

\NearlyCanonical

\begin{proof}
	
The hypotheses tell us that the action of $G_0$ on $T$ factors through an action of $\Aut(G_0)$ on $T$:
$$
\begin{tikzcd}
	\Aut(G_0) \arrow{r} & \Isom(T) \\
	G_0 \arrow{u}{g \mapsto \Ad(g)} \arrow{ur}{\cdot}
\end{tikzcd}
$$

We let $A$ denote the subgroup of $\Aut(G)$ which preserves $G_0$ setwise. The restriction map (which in general is neither injective nor surjective) gives us a homomorphism from $A$ to $\Aut(G_0)$, and so $A$ acts on $T$ via this map (and the previous action of $\Aut(G_0)$). 
$$
\begin{tikzcd}
	A \arrow{r}{res} & \Aut(G_0) \arrow{r} &  \Isom(T) \\
\end{tikzcd}
$$

 Since $G_0$ is normal, $\Inn(G)$ is a subgroup of $A$, and so this action defines an action of $G$ on $T$. 
$$
\begin{tikzcd}
	A \arrow{r} & \Aut(G_0)  \arrow{r} & \Isom(T) \\
	G \arrow{u}{g \mapsto \Ad(g)} 
\end{tikzcd}
$$

In particular, with respect to this action, $A \leq \Aut^T(G)$. Moreover, since $G_0$ is finite index in $G$, $A$ is a finite index subgroup of $\Aut(G)$ and hence the action of $G$ on $T$ is nearly canonical.

Finally, this action of $G$ on $T$ extends the original action of $G_0$ since the following diagram commutes (the left two maps from $G_0$ are just the maps sending a group element to the inner automorphism it defines, and the rightmost map is the one given by the original action of $G_0$): 
\[
\begin{tikzcd}
	A \arrow{r} & \Aut(G_0) \arrow{r} & \Isom(T) \\ \\
	G_0 \arrow{uu} \arrow{uur} \arrow{uurr}{\cdot}
\end{tikzcd} \qedhere
\]
\end{proof}

\begin{rem}
	One can clearly weaken the hypothesis in the Proposition above so that $T$ is only nearly $G_0$ canonical, and essentially the same proof works. However, the normality of $G_0$ seems essential to get a $G$-action. If $G_0$ were not normal, one could pass to a further finite index subgroup, $H$, of $G_0$ which would be normal in $G$. But then the action of $H$ on $T$ has no reason to be canonical or nearly canonical. The example of $\Z^2$ in the Klein bottle group shows that passing to a finite index subgroup is not a benign process from this point of view. 
\end{rem}

\subsection{Automorphisms which preserve a splitting, and a theorem of Bass--Jiang}
Our proof strategy is to use trees of cylinders to produce a tree where enough of the (outer) automorphisms act, and then to analyse that subgroup. (There are some shortcuts when the defining automorphism is exponentially growing, and we do not have to do all the work ourselves.)

There is a thorough discussion of the structure of the group $\Out^T(G)$ of outer automorphisms that preserve an action on a tree in \cite{BassJiang}.

We recall below the main structural theorem of that paper. Note though that to save on notation we do not state the result in full. (To be precise, their result allows for a centre, although the filtration becomes a step longer. Also, they give a precise description of the quotients at (4) and (5).)

\begin{thm}[{\cite[Theorem 8.1]{BassJiang}}]
	\label{thm:BassJiang}
	Suppose a centreless group $G$ acts on a tree $T$, minimally and irreducibly. 
	Write $\Gamma$ for the quotient graph, and $\mathcal{O}$ for a (fixed) choice of orientation of the edges of $\Gamma$.
	Suppose $\Out^T(G)$ is the subgroup of $\Out(G)$ which acts on $T$ -- that is, preserves the length function of the action.
	Then there is a filtration of $\Out^T(G)$, \[\Out^T(G) \trianglerighteq \Out^T_0(G) \trianglerighteq \mathcal{T}^+ \trianglerighteq \mathcal{T} \trianglerighteq K \trianglerighteq 1\]
	The quotients at each stage are as follows: \begin{align}
		\Out^T(G) / \Out^T_0(G) &\leq \Aut(\Gamma) \\
		\Out^T_0(G) / \mathcal{T}^+ &\cong {\prod_{v \in V(\Gamma)}}'  \Out(G_v;\{G_e\}_{\iota(e)=v}) \\
		\mathcal{T}^+ / \mathcal{T} &\cong \prod_{e \in \mathcal{O}} \frac{\Ad(N_e) \cap \Ad(N_{\overline{e}})}{\Inn(G_e)} \\
		\mathcal{T}/K \quad & \text{is a quotient of } \prod_{e \in E(\Gamma)} C_{G_\iota(e)}(G_e)\\
		K \quad & \text{is a quotient of } {\prod_{e  \in \mathcal{O}}}Z(G_e)
	\end{align}
\end{thm}

The ``prime'' on the product at (2) indicates that it is restricted to elements where for every incident edge $e_0$ the induced outer automorphism of $G_{e_0}$ is also induced by an element in $\Out(G_{\tau(e_0)};\{G_e\}_{\iota(e)=\tau(e_0})$. 

This property is characterised by the following commutative diagram. Suppose $(\Theta_v)$, with $v$ ranging through the vertices of $\Gamma$, is an element of the product \[{\prod_{v \in V(\Gamma)}}  \Out(G_v;\{G_e\}_{\iota(e)=v}).\] Then $(\Theta_v)$ is an element of the restricted product if and only if for every edge $e$, with $v=\iota(e)$ and $w=\tau(e)$ there are representatives $\theta_v$ and $\theta_w$ of the relevant outer automorphisms (of $G_v$ and $G_w$), and an automorphism $\psi$ of $G_e$ so that both squares commute.

\[\begin{tikzcd}
G_v \arrow{d}[swap]{\theta_v} &   G_e \arrow{l}[swap]{\alpha_e} \arrow{r}{\alpha_{\overline{e}}} \arrow{d}{\psi}   &  G_w \arrow{d}{\theta_w}\\
G_v   &  G_e  \arrow{l}{\alpha_e} \arrow{r}[swap]{\alpha_{\overline{e}}}  &   G_w \end{tikzcd}\]

There is another exposition in \cite{LevittAutosHyperbolic}, from where we have borrowed some notation (for example, $\mathcal{T}^+$ is Levitt's bi-Twists).

\begin{ex}
	\label{terms}
	Let us explain the quotients in Theorem~\ref{thm:BassJiang}, and illustrate them with an example. 
	
	Use the notation as above, and let $\Phi \in \Out^T(G)$. Then this (outer) automorphism falls into some term of the filtration, and each term has a geometric meaning in terms of the action of $\Phi$ on $T$. 
	
	\begin{enumerate}[(1)]
		\item $\Phi$ is non-trivial in the first term if it induces a non-trivial automorphism of the quotient graph, $\Gamma = T/G$. 
		\item If $\Phi$ is trivial in the first term, then it induces an (outer) automorphism at every vertex group (with some extra compatibility conditions). It is then non-trivial in the second term if it induces a non-trivial outer automorphism at some vertex group. 
		\item The third term consists of \textit{bi-Twists}. These are trivial in the first and second terms, but induce automorphisms of the edge groups which are non-trivial outer automorphisms. 
		\item The fourth term consists of \textit{Twists}. The subgroup $\mathcal{T}$ appearing in Theorem~\ref{thm:BassJiang} is the same subgroup as the $\mathcal{T}$ of Theorem~\ref{thm:rel_hyp_split}. These are trivial in the preceding terms, inducing the trivial outer automorphism on both vertex and edge groups. However, to be non-trivial in the quotient $\mathcal{T}/K$, the conjugations induced cannot be realised by elements of the centre of the edge group.
		\item Finally, we get the \textit{Dehn Twists} (this matches our terminology for free groups, when the tree is an $F_n$ tree with maximal cyclic edge stabilisers). These are induced by conjugations by elements of the centres of the edge groups. 
	\end{enumerate}

In many situations, various of these terms may be avoided by passing to a finite index subgroup of $\Out^T(G)$. This is the case in Theorem~\ref{thm:rel_hyp_split}, where the subgroup $\Out^1(G;\mathcal{P})$ is a finite index subgroup which induces trivial graph automorphisms and so that bi-Twists are all, in fact, Twists. The subgroup, $\mathcal T$ is then the group of twists. 

Let us illustrate this with example where we can find non-trivial automorphisms in all the terms above. We will take an amalgamated free product of two Klein bottles, glued together over an infinite cyclic group which is not maximal. 

Concretely, 
$$
\begin{array}{rcl}
	K_1 & = & \langle a_1, t_1 \ : \ a_1^{t_1} = a_1^{-1} \rangle \\ \\
	K_2 & = & \langle a_2, t_2 \ : \ a_2^{t_2} = a_2^{-1} \rangle \\ \\
	G & = & K_1*_{\langle a_1^2=a_2^2 \rangle } K_2.
\end{array}
$$

We then realise $G$ as the fundamental group of a graph of groups with one edge (whose edge group is $\langle a_1^2 \rangle = \langle a_2^2 \rangle \cong \Z$) and two vertices (whose vertex groups are $K_1$ and $K_2$ respectively). Let $T$ be the corresponding Bass-Serre tree. Consider the following automorphisms of $G$. 

\begin{center}
\begin{tabular}{ccccc}
$
\begin{array}{ccc}
	& \Phi_1 &  \\
	a_1 & \mapsto & a_2 \\
	t_1 & \mapsto & t_2 \\
	a_2 & \mapsto & a_1 \\
	t_2 & \mapsto & t_1 \\
\end{array}
$
& 

&

$
\begin{array}{ccl}
	& \Phi_2 &  \\
	a_1 & \mapsto & a_1 \\
	t_1 & \mapsto & t_1^{-1} \\
	a_2 & \mapsto & a_2 \\
	t_2 & \mapsto & t_2 \\
\end{array}
$

 & 
 
&
 $
 \begin{array}{ccl}
 	& \Phi_3 &  \\
 	a_1 & \mapsto & a_1^{t_1} = a_1^{-1} \\
 	t_1 & \mapsto & t_1^{t_1} = t_1 \\
 	a_2 & \mapsto & a_2^{t_2} = a_2^{-1} \\
 	t_2 & \mapsto & t_2^{t_2} = t_2 \\
 \end{array}
 $
\end{tabular}

\begin{tabular}{ccc}

$
\begin{array}{ccl}
	& \Phi_4 &  \\
		a_1 & \mapsto & a_1^{a_1} = a_1 \\
	t_1 & \mapsto & t_1^{a_1} = t_1 a_1^2 \\
	a_2 & \mapsto & a_2\\
	t_2 & \mapsto & t_2 \\
\end{array}
$

& &

$
\begin{array}{ccl}
	& \Phi_5 &  \\
	a_1 & \mapsto & a_1^{a_1^2} = a_1 \\
t_1 & \mapsto & t_1^{a_1^2} = t_1 a_1^4 \\
a_2 & \mapsto & a_2\\
t_2 & \mapsto & t_2 \\
\end{array}
$
\end{tabular}
\end{center}

Each $\Phi_i$ is an element of $\Out^T(G)$ which is non-trivial in the $i^{th}$ term of the filtration. 

The automorphism $\Phi_1$ swaps $K_1$ and $K_2$, so inverts the edge in $T/G$ and is non-trivial in the first term of the filtration. $\Phi_2$ induces a non-trivial outer automorphism of $K_1$, so is non-trivial in the second term. $\Phi_3$ is inner (a trivial outer automorphism) on both $K_1$ and $K_2$ but induces a non-trivial outer automorphism of the edge group, so is a bi-Twist.  $\Phi_4$ is a Twist, as it is inner on vertex and edge groups, but the conjugation on $K_1$ cannot be realised as conjugation by an element of the centre of the edge group, and finally $\Phi_5$ is a Dehn Twist. 
\end{ex}

Our common strategy for the polynomially growing case is to construct a canonical tree -- possibly only truly canonical for a finite index subgroup -- as a tree of cylinders, and then use this theorem to analyse the automorphisms which preserve it.

By Lemma~\ref{lem:stabs_are_free_by_cyclic} the quotient graph for the action must be finite, and so the quotient at (1) will be finite in every case. The quotient at (2) contains the McCool groups, which are generally easier to analyse.

The following lemma relates the restricted product at (2) in Theorem~\ref{thm:BassJiang} to the McCool groups (see Definition~\ref{def:mccool_etc}) for the vertex groups with respect to their incident edge groups. It is analogous to part of Proposition~2.3 of~\cite{LevittAutosHyperbolic} which deals with the case where $\Out(G_e)$ is finite. 

\begin{lem}
\label{lem:quotient_by_McCool}
The product $\prod_{v \in V(\Gamma)} \MC(G_v;\{G_e\}_{\iota(e)=v})$ is a normal subgroup of the restricted product ${\prod_{v \in V(\Gamma)}}'  \Out(G_v;\{G_e\}_{\iota(e)=v})$. 

The quotient is isomorphic to a subgroup of $\prod_{e \in E(\Gamma)} A_e/\Ad(N_e)$, where $A_e$ is a subgroup of $\Aut(G_e)$, every element of which is induced by an automorphism of $G_v$. Further, $A_e=A_{\overline{e}}$ for all edge pairs $\{e,\overline{e}\}$.
\end{lem}

\begin{proof}

For each edge at a vertex $v$ there is a map from $\Out(G_v;\{G_e\}_{\iota(e)=v})$ to \linebreak$\Aut(G_e)/\Ad(N_e)$ (note that this is a quotient of $\Out(G_e)$). Assembling them, we get a map to their product, and the kernel of this map consists of those elements induced by conjugations at every vertex; precisely the McCool group $\MC(G_v;\{G_e\}_{\iota(e)=v})$. 
The conditions on the initial restricted product amount to requiring that an element induces the same automorphisms on the stabiliser of an edge and its inverse: that is, the automorphisms $A_e$ of $G_e$ and $A_{\overline{e}}$ of $G_{\overline{e}}$ will be the same. (Though note that the quotient $\Aut(G_e)/\Ad(N_e)$ depends also on the vertex group, and so there is no reason to expect these will be the same for both an edge and its opposite.)
\end{proof}


Our strategy is to prove that the McCool groups are finitely generated, and that the quotient is too, usually by showing that this is true of every subgroup of this product. The details vary and appear in the relevant case.

In most of our cases, the edge groups are virtually abelian (that is, their free part has rank at most 1). In this case, we can understand the quotient at (3) as well.

\begin{prop}
	\label{prop:normalisers}
	Suppose $G$ is free by cyclic, and $H \leq G$ is free-by-cyclic and virtually abelian. Then $N_G(H)$ induces a finite subgroup of $\Out(H)$. (That is, $\Ad(G,H)/\Inn(H)$ is finite.)
\end{prop}
\begin{proof}
	The hypotheses give us that $H$ must be trivial, $\Z$, $\Z^2$ or the fundamental group of a Klein bottle. In every case except for $\Z^2$, the outer automorphism group is finite and so there is nothing to prove. So suppose $H=\Z^2$. The kernel of the map to $\Out(H)$ contains $H$, so it is enough to show that $H$ is finite index in $N_G(H)$. $H \cap F_n$ must be infinite cyclic, and we have that $H\cap F_n \leq N_G (H) \cap F_n \leq N_{F_n}(H\cap F_n)$. (Recall that conjugating cannot change the exponent of the stable letter.) Since the leftmost group is finite index in the rightmost group, it is also finite index in the middle group.
	
	In the quotient, both $H$ and $N_G(H)$ have non-trivial image. So the image of $H$ is finite index in the image of $N_G(H)$. The index of $H$ in $N_G(H)$ is the product of these two indices, and is therefore finite too.
\end{proof}

To show that the quotient at (4) is finitely generated, we will show that the centralisers (and therefore any quotient of their product) are finitely generated. The splittings we define for the polynomial case all have edge and vertex groups with finitely generated free part, so we will use the following lemma.

\begin{lem}
\label{lem:centralisers_are_fg}
	Suppose $H \leq G$ are (finitely generated free)-by-cyclic. Then $C_G(H)$ is finitely generated.
\end{lem}
\begin{proof}
	Let $F_n$ be the ``free part'' of $G$, the kernel of the given map to $\Z$. If $H\cap F_n$ is rank at least two, then $C_G(H)\cap F_n$ is trivial, and so $C_G(H)$ is either trivial or $\Z$. If $H\cap F_n$ is $\Z$, then so is $C_G(H)\cap F_n$, and $C_G(H)$ may be $\Z$ or $\Z^2$. If $H\cap F_n$ is trivial, then $C_G(H)\cap F_n$ consists of those elements in $G\cap F_n$ which are fixed by conjugating by $H$. As the fixed subgroup of an automorphism of a free group, this is finitely generated (by Theorem~\ref{Bestvina-Handel}). The full centraliser has an additional generator which is a root of the generator of $H$.
\end{proof}

The centre of a free-by-cyclic group $G$ is isomorphic to $\Z^2$ if and only if $G$ is; $\Z$ if $G$ is virtually free-times-cyclic and not $\Z^2$, and trivial otherwise. So the group given at (5) is a finitely generated abelian group, as are all its quotients.

In the exponential case, $G$ is a one-ended relatively hyperbolic group. We are able to use previous work in the literature (\cite{GuirardelLevitt2015}) on canonical JSJ decompositions and the automorphisms that preserve them.

\section{Exponential growth}
\label{sec:exp_growth} 

\subsection{Relative Hyperbolicity}

In this section, we assume that $\varphi$ is exponentially growing. Then we have access to a very useful fact: the group $G \cong F_n \rtimes_\varphi \Z$ is relatively hyperbolic (see \cite{GauteroLustigRelHyp,GhoshRelHyp,DahmaniLiRelHyp}). Several definitions of relative hyperbolicity, together with proofs of their equivalence, can be found in \cite{HruskaRelHyp}, for instance; we do not include one here since we do not work directly with the definition. 

Given a free group outer automorphism $\Phi$, say a subgroup $P$ is \emph{polynomially growing} (for $\Phi$) if there is a power $m$ and a representative $\alpha$ of $\Phi^m$ so that $P\alpha = P$ and the restriction of $\alpha$ to $P$ is polynomially growing.

\begin{prop}[{\cite[Proposition 1.4]{LevittInequalities}}]
	Every non-trivial polynomially growing subgroup is contained in a unique maximal polynomially growing subgroup. Maximal polynomially growing subgroups have finite rank, are malnormal, and there are only finitely many conjugacy classes of them.
\end{prop}

These maximal polynomially growing subgroups are a key ingredient in the relatively hyperbolic structure of a free-by-cyclic group:

\begin{thm}[\cite{GauteroLustigRelHyp,GhoshRelHyp,DahmaniLiRelHyp}]
	\label{thm:rel_hyp}
	If $\varphi$ is an automorphism of $F_n$ with at least one exponentially growing element, the semidirect product $F_n \rtimes_\varphi \Z$ is relatively hyperbolic with respect to 
	subgroups of the form $H \rtimes_{\varphi^m\gamma} \Z$, where $H$ is a maximal polynomially growing subgroup, $m$ is the minimum (positive) power of $\varphi$ which carries it to a conjugate, and $\gamma$ is the inner automorphism so $\varphi^m\gamma$ preserves $H$.
\end{thm}

(This collection is sometimes referred to as the ``mapping torus'' of the collection of maximal polynomially growing subgroups. For each $H$, that such an $m$ exists is guaranteed since there are only finitely many conjugacy classes of maximal polynomially growing subgroups, and since $\varphi^m$ sends $H$ to a conjugate, there is an inner automorphism so that the composition preserves $H$.) 

Recall that Lemma~\ref{lem:stabs_are_free_by_cyclic} gives that $F_n \rtimes \Z$ is one ended, and so it is one ended relative to any collection of subgroups.

Now we have access to a wide range of technology used in the study of relatively hyperbolic groups. In \cite{GuirardelLevitt2015} there is a careful examination of the subgroup $\Out(G;\mathcal{P})$ for relatively hyperbolic groups which are one ended relative to their parabolic subgroups, using JSJ theory and analysing the subgroup of automorphisms which preserves a splitting. We recall enough of their work to make the statements which follow self contained (although the proofs will not be).

There is a JSJ decomposition space over elementary (parabolic and virtually cyclic) subgroups relative to $\mathcal{P}$, which is invariant under $\Out(G;\mathcal{P})$. It contains a canonical JSJ tree, the tree of cylinders of the deformation space, which again is $\Out(G;\mathcal{P})$-invariant. There are four possibilities for vertex stabilisers: \begin{description}

\item[Maximal loxodromic] stabilised by a maximal virtually infinite cyclic group
\item[Maximal parabolic] stabilised by a maximal parabolic subgroup
\item[Rigid] non-elementary and elliptic in every splitting with elementary edge groups and where $\mathcal{P}$ are elliptic
\item[Flexible QH with finite fiber] none of the above, in which case they are ``quadratically hanging with finite fiber''
\end{description}

Since the groups we consider are torsion free, the maximal loxodromic subgroups are infinite cyclic and there can be no ``Flexible QH with finite fiber" vertex stabilisers. In general, these groups map with finite kernel onto an orbifold group, and the incident edge groups are virtually cyclic (and their images are in boundary subgroups). Since we are considering groups which are torsion free, the structure is actually much simpler here. First, the kernel must be trivial, so the group itself is an obifold group. By \cite[Lemma 2.4]{Dahmani2016} this is (virtually) free and hence infinitely ended, and therefore cannot occur as a vertex group with the required (virtually) cyclic incident edge groups by Lemma~\ref{lem:stabs_are_free_by_cyclic}. 

The tree is bipartite: one class of vertices is those stabilised by a maximal elementary group, and the other is the rigid vertices. Edge groups are maximal elementary subgroups of the rigid vertex group they embed in.

As Guirardel and Levitt point out, Lemma 3.2 of \cite{MinasyanOsin2012} tells us that when the groups in $\mathcal{P}$ are not themselves relatively hyperbolic, every automorphism permutes the conjugacy classes of the $P_i$. This is true in our case.
Theorem~\ref{thm:rel_hyp_split} below concerns $\Out(G;\mathcal{P})$; since this consists of those (outer) automorphisms which \emph{preserve} each of these conjugacy classes, it is a finite index subgroup of $\Out(G)$. 

Before we state the theorem, we define the group of twists, a subgroup of automorphisms of $G$. (See Section 2 of \cite{LevittAutosHyperbolic} or Subsection 2.6 of \cite{GuirardelLevitt2015}.)

\begin{defn}
	Let $e$ be an edge of a graph of groups, and $g$ an element of $C_e$ (the centraliser of $G_e$ in $G_{\iota(e)}$). Define the \emph{twist by $g$ around $e$} to be the automorphism that: \begin{itemize}
		\item if $e$ is separating, so $G=A\ast_{G_e}B$, conjugates $A$ by $g$ and fixes $B$ (with $B$ corresponds to the factor containing $G_{\iota(e)}$);
		\item if $e$ is non-separating, so $G=A\ast_{G_e}$, fixes $A$ and sends the stable letter $t$ to $tg$.
	\end{itemize}
	The \emph{group of twists}, $\mathcal{T}$, is the group generated by all twists.
\end{defn}

The group of twists $\mathcal{T}$ is a quotient of the direct product of all $C_e$, the centralisers of edge groups in adjacent vertex groups. Also, recall that McCool groups are defined in Definition~\ref{def:mccool_etc}. 

\begin{thm}[{\cite[Theorem 4.3]{GuirardelLevitt2015}}]
\label{thm:rel_hyp_split}
Let $G$ be hyperbolic relative to $\mathcal{P} = \{P_1,\dots,P_n\}$, with $P_i$ infinite and finitely generated, and assume that $G$ is one-ended relative to $\mathcal{P}$. Then there is a finite index subgroup $\Out^1(G;\mathcal{P})$ of $\Out(G,\mathcal{P})$ which fits into the exact sequence \[
1 \to \mathcal{T} \to \Out^1(G;\mathcal{P})\to \prod_{i=1}^p\MCG^0_{T_{can}}(\Sigma_i) \times \prod_j \MC(P_j;\Inc(P_j)) \to 1 \]
where $T_{can}$ is the canonical JSJ decomposition relative to $\mathcal{P}$, $\mathcal{T}$ is its group of twists; $\MCG^0_{T_{can}}(\Sigma_i)$ relate to flexible vertex groups; and $\MC(P_j;\Inc(P_j))$ is the McCool group of $P_j$ with respect to the incident edge groups. (The product is taken only over those parabolic subgroups which appear as vertex stabilisers in $T_{can}$)
\end{thm}

%
%

(Theorem~\ref{thm:rel_hyp_split} is derived from Levitt's discussion in \cite{LevittAutosHyperbolic}, together with some analysis of the bitwists,  showing that they are all twists, and extended McCool groups that can appear, to deduce that there is a finite index subgroup fitting into this short exact sequence. Compared to the Bass-Jiang approach, they show that the second normal subgroup is just $\mathcal{T}$ and that the first quotient has a finite index subgroup isomorphic to right hand term above.

We also note that the subgroup $\mathcal{T}$ is the group of `Twists', as appears in term (4) of the Bass-Jiang filtration, Theorem~\ref{thm:BassJiang} - see Example~\ref{terms} for an explanation and example of this.)

In our case there are no flexible vertex groups, so that term does not appear.
We will use this theorem to prove finite generation for $\Out^1(G;\mathcal{P})$, which will give us finite generation of $\Out(G)$. This will follow from showing that the group of twists and the McCool groups which can appear are finitely generated.

Levitt in \cite{LevittInequalities} provides several inequalities relating invariants of an outer automorphism. Theorem 4.1 of that paper concerns the ranks of conjugacy classes of maximal polynomially growing subgroups for an automorphism of $F_n$ and gives that it is at most $n-1$ when the automorphism is exponentially growing (since there is at least one exponentially growing stratum in this case).

\begin{prop}
\label{prop:twists_for_exponential}
The group of twists is finitely generated.
\end{prop}
\begin{proof}
The group of twists is a quotient of the direct product of the centralisers of the edge groups in the vertex groups so it is enough to show that all of these are finitely generated. The edge and vertex groups have the structure of a free-by-cyclic group: say the vertex group is $V=F \rtimes \langle t^kg \rangle$, and the edge group is $E = H \rtimes \langle t^\ell h \rangle$, where $H = E \cap F$. (Note that $H$ and $F$ are not necessarily finitely generated but are subgroups of the defining free group, which is.) If $H$ has rank at least two, then its centraliser in $F$ is trivial, and so the centraliser of $E$ in $V$ is at most infinite cyclic. If $H$ is infinite cyclic, then so is its centraliser in $F$; then the whole centraliser is either $\Z$ or $\Z^2$.

The final case is where $H$ is trivial, so we are interested only in the centraliser of $t^\ell h$. Again, it will be sufficient to show that the centraliser in $F$ is finitely generated, since there is at most one more generator contributed from the ``cyclic part'' to the full centraliser. The argument is different at rigid and maximal elementary vertex groups.

First consider rigid vertex groups. Since conjugating by $t^\ell h$ induces the automorphism $\varphi^\ell \Ad(h)$, any $w$ in $F$ that commutes with $t^\ell h$ is fixed by $\varphi^\ell \Ad(h)$. In particular, it is polynomially growing for the outer automorphism $\Phi$. This implies that $\langle w, t^\ell h \rangle$ is an elementary subgroup. Since edge groups are maximal elementary in rigid vertex groups, this cannot happen and so there is no such $w$. (For the same reason, there is no root of $t^\ell h$.) 

At maximal elementary vertices, the free part of the centraliser is the fixed subgroup for the automorphism of $F$ induced by conjugating by $t^\ell h$ (again, conjugation induces the automorphism $\varphi^\ell \Ad(h)$, so any element of $F$ that commutes is fixed by this automorphism). Since $F$ is finitely generated (as a maximal polynomially growing subgroup), so is this fixed subgroup (in fact the rank is bounded by the rank of $F$; Theorem~\ref{Bestvina-Handel}). 
\end{proof}

Thus far what we have said is true for any finitely generated free group; but we do not (yet) have the tools to understand McCool groups of free-by-cyclic groups in general. So we specialise to $F_3$, for the sake of Theorem~\ref{thm:rank3}.

In this case, the bounds on polynomially growing subgroups mean they can have rank at most 2. Here we can analyse the McCool groups, since there is a good classification of the outer automorphism groups for rank 2 in \cite{Rank2Case}, and rank 1 is fairly easy to understand.

\begin{prop}
\label{prop:McCool_for_exponential}
Suppose $G=F_n \rtimes \Z$, with $n=1,2$. Let $\mathcal{H}$ be a finite collection of finitely generated subgroups of $G$. Then $\MC(G;\mathcal{H})$ is finitely generated.
\end{prop}
\begin{proof}
In rank 1 the outer automorphism groups are $\GL_2(\Z)$ or finite, and in both cases this subgroup must be finitely generated. (For $\Z^2$ notice that elements are their own conjugacy classes, and if $g$ is fixed, so is its root, and so after changing basis the only matrices in the subgroup are triangular, and so it is virtually cyclic.)

In rank 2, we refer to \cite[Theorem~1.1]{Rank2Case} for their outer automorphism groups. Most cases are either finite or virtually cyclic: so any subgroup is finitely generated. The remaining cases are $G= F_2 \times \Z$, and $G = F_2 \rtimes_{-I_2} \Z$.

In the first of these, we have that $\Out(G)=(\Z^2 \rtimes C_2) \rtimes \GL_2(\Z)$ \cite[Theorem 1.1(i)]{Rank2Case}. Since $\GL_2(\Z)$ preserves each of the first two factors, we may pass to a finite index subgroup that is $\Z^2 \rtimes \GL_2(\Z)$. (An element $u \in \Z^2$ acts by sending $t^kg \to t^{k+u\cdot g_ab}g$, and $\GL_2(\Z)$ on the free part as you might expect.) Now consider a set of finitely generated subgroups $\mathcal{H}$. 

Since $t$ is central, its exponent cannot be changed by inner automorphisms. So any element of the McCool group must fix the $t$-exponent in each generator: this will give a subgroup of $\Z^2$ (orthogonal to the abelianised free parts of the generators) which is therefore finitely generated. So our McCool group is finitely generated if and only if its intersection with $\Out(F_n)$ is. In fact, this intersection is exactly the McCool group for the free part: since $t$ is central, it cannot identify any conjugacy classes of $F_n$. These are finitely generated by \cite{McCool1975}, which completes the proof. Note that McCool proves the result for elements; however in the free group case and more generally for toral relatively hyperbolic groups~\cite[Corollary 1.6]{GuirardelLevittMcCoolGroups} the McCool group for a finite set of subgroups is equal to the McCool group for some finite set of elements. 


For $F_2 \rtimes_{-I_2} \Z$, the outer automorphism group is $\PGL_2(\Z) \times C_2$ \cite[Theorem 1.1(ii)]{Rank2Case}. Again, we can just consider the finite index subgroup $\PGL_2(\Z)$, which only acts on the free part. We can consider the McCool group for the free group (as a subgroup of $\GL_2(\Z)$). Its image in $\PGL_2(\Z)$ is a finite index subgroup of the subgroup we want, which is therefore finitely generated. 
\end{proof}

We now summarise this case in a theorem.

\begin{thm}
\label{thm:exponential}
Suppose $G \cong F_3 \rtimes_\varphi \Z$, and $\varphi$ is exponentially growing. Then $\Out(G)$ is finitely generated.
\end{thm}

\begin{proof}
Use the canonical tree and the analysis of the outer automorphisms derived from it in Theorem~\ref{thm:rel_hyp_split}. Propositions~\ref{prop:twists_for_exponential} and~\ref{prop:McCool_for_exponential} show that the outside groups in the short exact sequence are finitely generated, and therefore so is $\Out_1(G;\mathcal{P})$ which is a finite index subgroup of $\Out(G)$.
\end{proof}

\section{Linear growth}

\subsection{Strategy}Our strategy for showing that the automorphism group of a free-by-cyclic group, in the case of linear growth, is as follows. 

\begin{itemize}
	\item Start with a free-by-cyclic group, $G= F_n \rtimes_{\Phi} \Z$, where $\Phi$ has linear growth,
	\item Consider a finite index subgroup, $G_0=F_n \rtimes_{\Phi^r} \Z$, so that $\Phi^r$ is UPG, and hence a Dehn Twist
	\item Use the parabolic orbits Theorem to find a tree whose deformation space is invariant,  
	\item Deduce that the tree of cylinders, $T=T_c$, of this space is $G_0$-canonical, 
	\item Use Proposition~\ref{prop:no_really_theres_an_action} to deduce that $T$ is nearly $G$-canonical
	\item Show that $\Out^{T}(G)$ is finitely generated if certain McCool groups for free-by-finite groups are
	\item Carry out the calculation of the relevant McCool groups, to conclude that $\Out^{T}(G)$ is finitely generated.
\end{itemize}

\subsection{Constructing a tree}
First we record a useful lemma on normalisers in free-by-cyclic groups.
 
\begin{lem}
	\label{lem:small_normaliser_Z_2}
	Suppose $F_n \rtimes \langle s \rangle$ is a free-by-cyclic group, and $w \in F_n$ is not a proper power and commutes with $s$. Then $\langle w,s \rangle$ is its own normaliser.
\end{lem}
\begin{proof}
	Suppose $s^kg \in F_n \rtimes \langle s \rangle$ so that $\langle w,s \rangle^{s^kg}=\langle w,s \rangle$. This gives that $\langle w^g, s^g \rangle =\langle w,s \rangle$. Taking intersections with $F_n$, we must have that $w^g \in \langle w \rangle$. But this means that $g \in \langle w \rangle$ so $s^kg \in \langle w,s \rangle$ as required.
\end{proof}

In the following Proposition, we take a Dehn Twist and use the Parabolic Orbits Theorem~\ref{thm:parabolic_orbits} to get a tree on which the corresponding free-by-cyclic group acts. We would like, at this stage, to say that the resulting action is canonical for the free-by-cyclic group. Although this seems plausible, our proof goes via the tree of cylinders construction which is guaranteed to be canonical and -- as we prove in this case -- remains in the same deformation space. 

\begin{prop}
\label{prop:canonical_tree_for_Dehn_twist}
	Suppose $\varphi$ is a UPG and linear automorphism of $F_n$. Then there is a canonical action of $G_0= F_n \rtimes_\varphi \Z$  on a tree, where \begin{enumerate}
	\item Edge stabilisers are maximal $\Z^2$;
	\item Vertex stabilisers are either maximal $\Z^2$, or maximal $F_m \times \Z$ with $n \geq m \geq 2$.
	\end{enumerate}
\end{prop}

\begin{proof}
The initial input for the construction is the Dehn twist, $\varphi$. By Theorem~\ref{thm:parabolic_orbits}, there is a unique simplicial $F_n$-tree (defining a simplex in the boundary of $CV_n$) that is preserved by $\varphi$. 
This tree gives a splitting of $F_n$, where the vertex stabilisers are fixed subgroups (of rank at least two) corresponding to different representatives of the outer automorphism, and the edge groups are maximal infinite cyclic. By Theorem~\ref{Bestvina-Handel} there are only finitely many conjugacy classes of these subgroups, and their ranks are bounded by $n$.

Since it is fixed by $\varphi$, the same tree provides a splitting for $G_0$. The vertex groups are now free times cyclic, and the edge groups are maximal $\Z^2$. (They are generated by the original edge group generator $g$, together with an element $sw$ in either adjacent edge group which commutes with $g$. They must be maximal since otherwise there would be another element $s^kh$ commuting with $g$ (and $sw$); writing this element as $(sw)^kh'$ implies that $h'$ commutes with $g$. Since $g$ generated a maximal infinite cyclic subgroup of $F_n$, $h'$ is a power of $g$, and so $(sw)^kh'$ is contained in $\langle g,sw \rangle$.) 

This tree defines a deformation space which is preserved by automorphisms, since the vertex stabilisers can be specified algebraically: they are precisely the centralisers of some $sw$, corresponding to an automorphism in the outer automorphism class of $\varphi$ with fixed subgroup having rank at least $2$. (Equivalently, they are the centralisers that contain a copy of $F_2 \times \Z$.) So they will be permuted by automorphisms of $F_n \rtimes_\varphi \Z$ and the deformation space must be preserved. 

We now have most of the tools to start constructing a tree of cylinders for this deformation space: it remains to specify the family $\mathcal{E}$ of allowed edge stabilisers, and the admissible equivalence relation. We will take $\mathcal{E}$ to be maximal 
$\Z^2$, and the equivalence relation to be equality. (It is easy to check this is admissible, since if $A \leq B$ are both maximal $\Z^2$ then we must have $A=B$).

Now we can calculate the cylinders. First, note that a cylinder may contain at most one edge from each edge orbit. If two edges in the same orbit have the same stabiliser, then there is an element outside the stabiliser which normalises it. However, Lemma~\ref{lem:small_normaliser_Z_2} shows that there is no such element.

This also means that a cylinder stabiliser must actually stabilise it pointwise: since it is a subgroup of $G_0$, it cannot permute edges in different orbits. So cylinder stabilisers are precisely the stabiliser of any (and every) edge in that cylinder. Every vertex is in multiple cylinders, so is also in the tree of cylinders. 

Cylinders are finite, and in particular bounded, so the tree of cylinders will lie in the same deformation space. It is already collapsed, since the edge stabilisers are still (maximal) $\Z^2$. 
\end{proof}

\begin{rem}
Note that an alternative construction of this canonical tree involves subdividing every edge and folding -- the effect of constructing the tree of cylinders is to change to original tree so that each vertex has at most one adjacent edge with a given stabiliser. There are examples where the tree of cylinders is not very small -- it has tripod stabilisers, so the construction has done something.

However, the (finite index) subgroup of automorphisms which does not permute the underlying graph of groups does act on the original limiting tree, since we can recover it by equivariantly collapsing some edges. This means that in our terminology the action on the limiting tree was itself nearly canonical, though it is not clear how to find a direct proof of this fact.

If a cylinder had only one edge, then it will have been subdivided -- allowing (if the endpoints are isomorphic) for the possibility of inversions. (If not, or if the endpoints are not isomorphic, no inversions are possible.)
\end{rem}

We now equip ourselves with a nearly canonical action for a general linearly growing automorphism , using this tree of cylinders.

\begin{prop}
\label{prop:linear_nearly_canonical_tree}
Suppose $G=F_n \rtimes_{\Phi} \Z$ is a free-by-cyclic group, and $\Phi$ is linearly growing. Then $G$ has a nearly canonical action on a tree $T$, where \begin{enumerate}
\item Edge stabilisers are virtually $\Z^2$ (and therefore either $\Z^2$ or the fundamental group of a Klein bottle).
\item Vertex stabilisers are $F_m \rtimes_\varphi \Z$ where $F_m$ is a subgroup of $F_n$, the rank $m$ is at most $n$, and $\varphi$ is a representative of $\Phi$, which restricts to and is periodic on $F_m$. (They are virtually free-times-cyclic.)
\end{enumerate}
\end{prop}

\begin{proof}
Since $\Phi$ has a power which is UPG, and therefore a Dehn twist, we pass to the normal finite index subgroup $G_0$ this suggests and use Proposition~\ref{prop:canonical_tree_for_Dehn_twist} to construct a canonical tree $T$. We then use Proposition~\ref{prop:no_really_theres_an_action} to extend this action to a nearly canonical action for $G$. Edge and vertex stabilisers in $G$ will contain edge and vertex stabilisers in $G_0$ as finite index subgroups, and must themselves be free-by-cyclic by Lemma~\ref{lem:stabs_are_free_by_cyclic}. Combining these properties gives the conclusions in (i) and (ii).
\end{proof}

\subsection{Reducing to free-by-finite groups}

We consider the subgroup $\Out^T(G)$ of outer automorphisms which preserves this tree, and apply Theorem~\ref{thm:BassJiang} to understand it. The quotients at parts (1) and (3-5) of the theorem are finitely generated by the observations following the theorem; the main difficulty is in understanding the quotient at (2).

First, we reduce to the case where we can consider McCool groups; we will then show that the result we want is implied by a similar result in the free-by-finite group obtained by quotienting by the centre, and in the next section prove the result there. (The arguments involved in the reduction and the following section are easier for the larger groups $\Out(G_v;\{G_e\}_{\iota(e)=v})$ at least when the edge groups all contain the centre of the vertex group as in our case. However, it does not seem possible to take account of the edge compatibility relations through this process, so we do need to pass to McCool groups.)

We begin 
with a straightforward structural result about free-by-cyclic groups defined by periodic outer automorphisms; 

\begin{lem}[{\cite[Proposition 4.1]{LevittRankOfGBS}}]
\label{lem:periodic_implies_centre}
Suppose $G$ is a free-by-cyclic group which is virtually free-times-cyclic and not virtually $\Z^2$. Then $G$ has an infinite cyclic centre, and is the fundamental group of a graph of groups with all edge and vertex groups isomorphic to $\Z$. 
\end{lem}

Such a group is known as a \emph{Generalised Baumslag-Solitar (GBS) group}, and having a non-trivial centre is equivalent to having trivial modulus, in the language of~\cite{LevittGBS}. The free-by-(finite cyclic) groups we will consider are obtained by taking a group of this kind and quotienting by the centre.

We now study the group appearing as a quotient at (2) in Theorem~\ref{thm:BassJiang}, beginning by considering automorphisms of edge groups that can be induced here.

\begin{lem}
\label{lem:induce_few_autos_of_edges}
Suppose $G$ is a free-by-cyclic group that is virtually free-times-cyclic, and $H_i$ is a collection of subgroups isomorphic either to $\Z^2$ or to the fundamental group of a Klein bottle. Then $\Out(G;\{H_i\})$ induces a virtually cyclic subgroup of $\Aut(H_i)$.
\end{lem}

\begin{proof}

If $H_i$ is the fundamental group of a Klein bottle, $\Out(H_i)$ is finite, and $\Inn(H_i)$ is virtually cyclic. Therefore $\Aut(H_i)$ is again virtually cyclic, and so is the subgroup induced by $\Out(G;\{H_i\})$.

If $H_i$ is $\Z^2$, it contains a finite index subgroup of the infinite cyclic centre of $G$. 
Let $\delta$ generate this subgroup. 
We can choose a basis $\{x_1,x_2\}$ for $H_i$ so that $\delta = x_1^k$ with $k>0$; roots are unique in $\Z^2$, so $x_1$ is as uniquely defined as $\delta$: it is unique up to inverses. Any automorphism of $G$ will preserve the centre; in particular it must send $\delta$ to itself or its inverse. So any automorphism restricting to $G_e$ will likewise send $x_1$ to itself or its inverse. Viewing elements of $\GL(2,\Z)$ as matrices, this implies that we can only induce automorphisms represented by triangular matrices. This subgroup is virtually cyclic.\qedhere
\end{proof}

We use this to characterise the subgroup generated when we quotient by the product of McCool groups, which will mean it is sufficient to prove that those are finitely generated.

\begin{prop}
\label{prop:linear_McCool_is_enough}
Suppose $G$ is a free-by-cyclic group where the defining outer automorphism is linearly growing. Let $T$ be the tree constructed in Proposition~\ref{prop:linear_nearly_canonical_tree}, with a nearly canonical action of $G$ (where edge stabilisers are virtually $\Z^2$, and vertex stabilisers are either virtually $\Z^2$ or virtually $F_m \times \Z$ with $m\geq2$). Then the quotient of ${\prod_{v \in V(\Gamma)}}'\Out(G_v;\{G_e\}_{\iota(e)=v})$ (from Theorem~\ref{thm:BassJiang} (2)) by $\prod_{v\in V(\Gamma)}\MC(G_v;\{G_e\}_{\iota(e)=v})$ (as described in Lemma~\ref{lem:quotient_by_McCool}) is finitely generated.
\end{prop}

\begin{proof}
We consider the projection to each factor $A_e/\Ad(N_e)$. The subgroup we are interested in is contained in the product of these projections, which we will show is Noetherian (every subgroup is finitely generated) and from there deduce that our subgroup must be finitely generated.

First, we consider the vertices where the stabiliser contains a rank 2 free group. In this case, by Lemma~\ref{lem:induce_few_autos_of_edges} each of these vertex groups can only induce a virtually cyclic subgroup of automorphisms of each edge group. This is a property closed under subgroups and quotients, so for every edge $e$ with $\iota(e)$ a vertex of this type the projection to $A_e/\Ad(N_e)$ is virtually cyclic.

The remaining vertices arose as cylinders, and their vertex groups are either the fundamental group of a Klein bottle or $\Z^2$ (as are the incident edge groups). If $G_v$ is a Klein bottle, then it has finite outer automorphism group. So $\Out(G_v;\{G_e\}_{\iota(e)=v})$ is finite, and $\Ad(N_e)$ must therefore be finite index in $A_e$ for each edge group. So the projection to $A_e/\Ad(N_e)$ for edges starting at these vertices is finite.

If $G_v$ is $\Z^2$, we need to use the structure of the tree. The quotient graph inherits the bipartite structure of the tree of cylinders constructed in Proposition~\ref{prop:canonical_tree_for_Dehn_twist} -- every edge joins a cylinder vertex to a vertex with larger stabiliser. By Lemma~\ref{lem:quotient_by_McCool} the induced automorphisms $A_e$ and $A_{\overline{e}}$ of the stabilisers of an edge and its inverse are the same. By Lemma~\ref{lem:induce_few_autos_of_edges} this is virtually cyclic, and so the same is true of the projection to $A_e/\Ad(N_e)$ in this case.

Assembling these projections we get a group that is virtually finitely generated abelian, and in particular is Noetherian.  So any subgroup -- including the quotient of 
$${\prod_{v \in V(\Gamma)}}'\Out(G_v;\{G_e\}_{\iota(e)=v})$$
 by the product of McCool groups -- is again finitely generated.
\end{proof}

In the Klein bottle case, the McCool group (as with any subgroup of the outer automorphism group) is finite, and in particular finitely generated. In the $\Z^2$ case the McCool group is trivial since elements of $\GL_2(\Z)$ are uniquely characterised by their action on a finite index subgroup of $\Z^2$: as soon as an edge group is fixed, so is the whole vertex group. Therefore the remainder of the work is at the vertices stabilised by some $F_m \rtimes \Z$, with $m \geq 2$.

This reduces the problem to calculating the McCool groups at each vertex. We use Levitt's work in \cite{LevittGBS} to further reduce the problem to McCool groups of free-by-(finite cyclic) groups.

By Lemma~\ref{lem:periodic_implies_centre}, the larger vertex groups $G_v$ are Generalised Baumslag-Solitar groups with trivial modulus. Levitt proves this theorem, which we use to enable us to understand $\Out(G)$ in terms of the outer automorphisms of a free-by-finite group.

\begin{thm}[see {\cite[Theorem 4.4]{LevittGBS}}] 
	\label{thm:Levitt_for_periodic}
	Suppose $G$ is a GBS group with trivial modulus, and let $H$ be the quotient of $G$ by its centre. Then there is a finite index subgroup $\Out_0(G)$ of $\Out(G)$ fitting into a split exact sequence \[
1 \rightarrow \Z^k \rightarrow \Out_0(G) \rightarrow \Out_0(H) \rightarrow 1 \]

where $k$ is the rank of the underlying graph, and $\Out_0(H)$ is a finite index subgroup of $\Out(H)$. The section of $\Out_0(H)$ fixes the centre of $G$.
\end{thm}

The $\Z^k$ subgroup should be thought of as $\Hom(\pi_1(\Gamma),Z(G))$: it acts by multiplying every ``HNN-like generator'' by an element of the centre. In fact, it is generated by Dehn Twists - the last term in Theorem~\ref{thm:BassJiang}.   
The subgroup $\Out_0(H)$ consists of (outer classes of) automorphisms which preserve the conjugacy classes of elliptic elements, and the image of a certain map $\overline{\tau}$ to some finite cyclic group. 

The map $\tau$ is initially defined as a map, $\tau:G \to \Isom^+(\R)$ (translations of $\R$) and we then observe that the image  is discrete, and 
therefore $\tau: G \to \Z$.

Following \cite{LevittGBS}, define $\tau$ on a generator of a vertex group, $x_v$, as the translation by $1/n_v$, where $x_v^{n_v}=\delta$ for some $n_v \in \Z$ and where $\delta$ is a generator of the centre of $G$. Such an $n_v$ always exists since if a group acts on a tree without fixing an end, its centre lies in the kernel of the action. In particular, $\delta$ is contained in every vertex group.  (So $\tau(\delta)$ is translation by 1.) On generators arising from edges, $\tau$ is defined to be the identity. This is enough to define $\tau$ on the whole group. Since the image of every element is translation by a multiple of $1/\ell$, where $\ell$ is the least common multiple of the $n_v$, it follows that the image is discrete and hence $\tau: G \to \Z$. 

Further define $\overline{\tau}$ by taking a quotient by the group generated by $\tau(\delta)$. That is, $\overline{\tau}$  is the natural map, $\overline{\tau}: G \to \frac{G}{\ker \tau \cdot \langle \delta \rangle}$.

This definition does not apply to the ``elementary'' GBS groups, $\Z^2$ and the fundamental group of a Klein bottle. These are distinguished among free-by-cyclic groups as being virtually $\Z^2$, and this property cannot occur in a free-by-cyclic group with underlying free group having rank at least 2. Since the groups we consider here (corresponding to non-cylinder vertices in the nearly canonical tree) do, this definition (and the following arguments) apply in sufficient generality for our use.


In some sense, $\tau$ is a ``better'' map to $\Z$ than the one arising from the presentation of $G$ as a free-by-cyclic group. First consider its kernel:

\begin{lem}
\label{lem:tau_has_nice_kernel}
The kernel of the map $\tau$ is a finitely generated free group.
\end{lem}

This follows from the computation of the relevant BNS invariants in \cite[Corollary 3.2]{CashenLevittBNS}, but can be proved by more elementary methods as follows: 

\begin{proof}
Consider the action of the kernel on the GBS tree $T$. This action is free -- non-trivial elements of vertex stabilisers are not in the kernel of $\tau$ -- and so the kernel is a free group. It remains to show it is finitely generated. To do this consider $\overline{\tau}$, defined by passing to the quotient by $\tau(\delta)$. 

Note that $\ker(\overline{\tau}) = \ker(\tau) \cdot \langle \delta \rangle$, where $\tau: G \to \Z$ and $\delta \not\in \ker(\tau)$. Hence, $G/  \ker(\tau) \cdot \langle \delta \rangle$ is a finite cyclic group, and hence  $\ker(\tau) \cdot \langle \delta \rangle$ is a finite index subgroup of $G$, and hence is finitely generated. 

Moreover, since  $\delta$ has non-trivial (and infinite order) image under $\tau$, the intersection of $\ker(\tau)$ and $\langle \delta \rangle$ is trivial, and hence $ \ker(\tau) \cdot \langle \delta \rangle$ is actually the direct product, $\ker(\tau) \times \langle \delta \rangle$.

Since $\ker(\tau)$ is a quotient of the finitely generated group, $\ker(\tau) \times \langle \delta \rangle$, it too is finitely generated.

%
%
\end{proof}

This lemma shows that the map $\tau$ fibres: it gives us another way to write $G$ as a free-by-cyclic group. Note that the rank of the free group may have changed, but since there is still a centre, the defining outer automorphism must still be periodic. 
(Sometimes, though not always it becomes periodic as an automorphism -- for example, using this construction it becomes apparent that the rank three free-by-cyclic group defined using the automorphism $a \mapsto b^{-1}c, b \mapsto a^{-1}c, c \mapsto c$ is isomorphic to $F_2 \times \Z$.)

By design, this new presentation as a free-by-cyclic group is very well behaved when applying $\tau$: the image under $\tau$ of any element is the exponent of the (new) stable letter. This exponent is preserved by conjugation, and (by considering the stable letter as a root of $\delta$) by the section of $\Out_0(H)$. So if an automorphism whose outer class is an element of $\Out_0(G)$ does not preserve the exponent on the stable letter, writing it in the normal form for a semidirect product will involve a non-trivial element of the $\Z^k$ subgroup given in the decomposition of Theorem~\ref{thm:Levitt_for_periodic}. Note that since the exponent is preserved by conjugation, this effect is constant across an outer class.

\begin{prop}
	\label{prop:periodic_mccool}
	Suppose that $G$ is a free-by-cyclic group that is virtually free-times-cyclic, and $\{G_i\}$ is a family of subgroups. Write $H$ for the quotient of $G$ by its centre, and let $H_i$ be the image of the subgroup $G_i$ under this quotient map. Let $s$ be the section of Theorem~\ref{thm:Levitt_for_periodic}. Then \[\Out_0(G) \cap \MC(G,\{G_i\})=(\Z^k \cap \MC(G,\{G_i\}))\cdot s(\Out_0(H) \cap \MC(H,\{H_i\})).\] In particular, it is finitely generated if and only if $(\Out_0(H) \cap \MC(H,\{H_i\}))$ is.
\end{prop} 

\begin{proof}
	Consider the split short exact sequence of Theorem~\ref{thm:Levitt_for_periodic}. The kernel of the quotient map applied to $\Out_0(G) \cap \MC(G,\{G_i\})$ will be $(\Z^k \cap \MC(G,\{G_i\}))$, a finitely generated abelian group. The image of $\Out_0(G) \cap \MC(G,\{G_i\})$ in $\Out_0(H)$ is contained in $\MC(H,\{H_i\})$, since subgroups which are conjugate in $G$ have conjugate images in $H$. To complete the proof, we must show that this accounts for all of $(\Out_0(H) \cap \MC(H,\{H_i\}))$. To do this consider the section: we need to show that every element lifts to an element of $\Out_0(G) \cap \MC(G,\{G_i\})$.
	
	For an element of $(\Out_0(H) \cap \MC(H,\{H_i\}))$ consider the collection of representatives $\alpha_i$, each fixing the subgroup $H_i$. Lemma 4.1 of \cite{LevittGBS} constructs the equivalent section for automorphism groups; one of the properties of the lift $\overline{\alpha}$ of $\alpha$ is that applying $\overline{\alpha}$ first does not alter $\tau$. So if $\alpha_i$ fixes $h$, and $g$ is any preimage of $h$, $\overline{\alpha}_i$ must send $g$ to $g\delta^k$. However, since $\tau$ must be unaltered, in fact $k=0$ and $g$ is fixed. So each $\overline{\alpha}_i$ fixes the subgroup $G_i$. The last thing to check is that they all represent the same outer automorphism. This is the case since inner automorphisms lift to inner automorphisms (by any preimage of the conjugator, as they differ by a central element). So any (indeed every) $\overline{\alpha}_i$ represents an element of $\MC(G,\{G_i\})$, which is contained in $\Out_0(G)$ since it is the image of the section of Theorem~\ref{thm:Levitt_for_periodic}.
\end{proof} 

So to show that the McCool groups we are interested in are finitely generated, we need to show the same for the relevant McCool groups of free-by-finite groups. In our situation the edge groups are virtually $\Z^2$, and a power of a generator is central in the vertex group, so in $H$ the image of each edge group becomes virtually infinite cyclic.  In this case, we can understand the McCool groups.

\subsection{McCool groups for free-by-finite groups}

The purpose of this section is to study the groups $\MC(H,\{H_i\})$, which will complete our proof in the linear growth case.

\begin{prop}
\label{prop:McCool_of_free-by-finite}
	Suppose $H$ is virtually free  
	and $\{H_i\}$ is a finite collection of virtually infinite cyclic subgroups. Then $\MC(H,\{H_i\})$ is finitely generated.
\end{prop}

First we use a result which allows us to understand the outer automorphisms of the extension by considering the centraliser of the finite cyclic subgroup. 


\begin{prop}
	\label{prop:spotting_the_centraliser}
	Let $H$ be a group, and $F$ a normal subgroup of $H$ with trivial centre.  Let $\Ad(h)$ represent the automorphism of $F$ induced by conjugating by $h$. Let $\Aut_H(F)$ be the subgroup of $\Aut(F)$ that commutes with $\Ad(H)$ up to inner automorphisms. That, is the subgroup defined by \[\{ \alpha \in \Aut(F) | [\Ad(h),\alpha] \in \Inn(F) \quad \forall h \in H\}.\] Further, let $N$ be the subgroup of $\Aut(H)$ which preserves $F$ and all its cosets (that is, it acts trivially on the quotient $H/F$).
	
	Then the restriction to $F$ sends $N$ isomorphically to $\Aut_H(F)$.
\end{prop}

\begin{proof}


First we consider the image of the restriction map. Suppose $\alpha$ is an element of $N$, and consider its restriction to $F$. For all $f\in F$, and $h \in H$, $f(\Ad(h)\alpha) =(f^h)\alpha= (f\alpha)^{(h\alpha)} = f(\alpha \Ad(h\alpha))$. This gives that, as automorphisms of $F$, $\Ad(h\alpha)=\alpha^{-1}\Ad(h)\alpha$. Since $\alpha$ preserves cosets of $F$, $h^{-1}(h\alpha) = f$, for some $f \in F$. But then $\Ad(f)=\Ad(h^{-1}(h\alpha))=\Ad(h)^{-1}\alpha^{-1}\Ad(h)\alpha$: the restriction of $\alpha$ to $F$ satisfies the commutator property defining $\Aut_H(F)$, and so the image in $\Aut(F)$ lies in this subgroup.

Next we show that the restriction map is a surjection to $\Aut_H(F)$. To do this, we construct an automorphism of $H$ with a given image in $\Aut_H(F)$. For any $\alpha \in \Aut_H(F)$, we have $\alpha^{-1}\Ad(h)\alpha = \Ad(h)\Ad(f_{h,\alpha})$ by the defining commutator property, where $f_{h,\alpha}$ is an element of $F$ depending on both $h$ and $\alpha$. Since $F$ is centreless, it has a unique element inducing any inner automorphism -- $f_{h,\alpha}$ is well defined. Extend $\alpha$ to a function $\overline{\alpha}$ defined on all of $H$ by setting $h\overline{\alpha}$ to be $hf_{h,\alpha}$. (On $F$, since $\alpha^{-1}\Ad(h)\alpha = \Ad(h\alpha)$ for inner automorphisms, $h\alpha = hf_{h,\alpha}$, so the restriction is indeed $\alpha$.) To see $\overline{\alpha}$ is an endomorphism, we need to check that $hkf_{hk,\alpha} = hf_{h,\alpha}kf_{k,\alpha}$.

Consider the following diagram: the squares all commute by the definition of $\Aut_H(F)$, the left hand triangle is a consequence of $\Ad$ being a homomorphism, and we are interested in the right hand triangle, whose commutativity follows from 
chasing the diagram (noting that the top map is an isomorphism). This gives that $\Ad(hkf_{hk,\alpha}) = \Ad(hf_{h,\alpha}kf_{k,\alpha})$, and by normality of $F$ this is equal to $\Ad(hkf_{h,\alpha}^kf_{k,\alpha})$. That is, we have that the unique element of $F$ inducing the correct inner automorphism is $f_{hk,\alpha}=f_{h,\alpha}^kf_k$, with which we get that $hkf_{hk,\alpha} = hf_{h,\alpha}kf_{k,\alpha}$. 

\[\begin{tikzcd}[row sep=huge, column sep = huge]
F \arrow[pos=0.3]{d}{\Ad(h)} \arrow{r}[swap]{\alpha} \arrow[bend right=30, swap]{dd}{\Ad(hk)} &  F \arrow[swap, pos=0.7]{d}{\Ad(hf_{h,\alpha})} \arrow[bend left=30]{dd}{\Ad(hkf_{hk,\alpha})}\\
F \arrow[pos=0.3]{d}{\Ad(k)} \arrow{r}[swap]{\alpha} &  F \arrow[swap, pos=0.7]{d}{\Ad(kf_{k,\alpha})} \\
F \arrow{r}[swap]{\alpha} & F \end{tikzcd}\]

To see $\overline{\alpha}$ is surjective, note that it is surjective on $F$, and for general $h$, we have $h=(h(f^{-1}_{h,\alpha}\alpha^{-1}))\overline{\alpha}$. 
To see injectivity, suppose $h\overline{\alpha}=1$, so $hf_{h,\alpha} = 1$. In particular, this means $h$ is an element of $F$; but on $F$ $\overline{\alpha}$ agrees with $\alpha$, which is an automorphism. So $h=1$, and $\overline{\alpha}$ is an element of $\Aut(H)$, restricting to $\alpha$ on $F$ as claimed.




Finally, we show that the restriction map $N \to \Aut_H(F)$ is injective. Denote by $K$ the kernel of the map $\Ad: H \to \Aut(F)$. Since $F$ has no centre, $K \cap F$ is trivial.

Suppose $\alpha$ lies in the kernel of the restriction map, so it fixes every element of $F$. Then for all $f \in F, h \in H$, we have that $f^h \in F$, so $f^{h\alpha} = (f\alpha)^{h\alpha}=(f^h)\alpha = f^h$. So the actions of $h\alpha$ and $h$ on $F$ are the same: that is, $h\alpha$ and $h$ lie in the same $K$-coset. 

So for all elements $h \in H$ we have that  $(h\alpha)^{-1}h$ lies in $K$. Since both automorphisms preserve cosets of $F$, in fact $(h\alpha)^{-1}h$ lies in $F \cap K$. But these groups intersect trivially, so $h\alpha = h$ for all elements $h$, $\alpha$ must be the identity, and so the restriction map has trivial kernel. \qedhere
\end{proof} 

We now specialise this general result to our current case of virtually free groups.

\begin{cor}
\label{cor:OutCentraliserEnough}
	Let $H$ be a finitely generated virtually free group, that is not virtually cyclic, and $F$ a normal finite index subgroup (with rank at least 2) of $H$. Then the subgroup $\Aut_H(F)$ of $\Aut(F)$ is isomorphic to a finite index subgroup of $\Aut(H)$ which preserves $F$, and where the isomorphism is given by the restriction map.
\end{cor}

(This Corollary is similar in spirit to~\cite{McCoolFiniteExtensions}, which deals with centralisers in $\Aut(F)$; ours looks at the preimage of centralisers in $\Out(F)$, and deals simultaneously with the splitting and non-splitting cases.)

\begin{proof}
	By Proposition~\ref{prop:spotting_the_centraliser}, the subgroup $\Aut_H(F)$ of $\Aut(F)$ is isomorphic to the subgroup $N$ of $\Aut(H)$. This subgroup preserves $F$ and all its cosets, and the restriction to $F$ provides the isomorphism, as required. To finish the proof, notice that since $H$ is finitely generated and $F$ is a finite index subgroup, $N$ must be finite index in $\Aut(H)$.
\end{proof} 

We want not just the outer automorphism group but the McCool group. The relevant result about $\Out(F_n)$ is the following theorem of Bestvina, Feighn and Handel.

\begin{thm}[{\cite[Theorem 1.2(3)]{BestvinaFeighnHandelMcCoolGroups}}]
\label{thm:BFHcentraliser}
	Suppose $Q$ is a finite subgroup of $\Out(F_n)$, and $\Out_Q(F_n)$ is its centraliser. Let $K_1,\dots K_n$ be a collection of conjugacy classes of finitely generated subgroups of $F_n$. Then the subgroup of $\Out_Q(F_n)$ fixing each $K_i$ is VF (in particular, is finitely presented).
\end{thm}	

Note that the conclusion we want is stronger: we want the action on a representative of $K_i$ to be by conjugation, not just sending it to a conjugate. However, as the relevant subgroups are infinite cyclic this is only a matter of passing to a finite index subgroup.

These theorems allow us understand the subgroup of outer automorphisms conjugating an element that lies in the finite index free subgroup; to extend the result to the full subgroups $H_i$, we need the following lemmas.

\begin{lem}
\label{lem:Out_virt_cyclic_finite}
Suppose $A$ is a virtually cyclic group. Then $\Out(A)$ is finite.
\end{lem}

See, for instance, \cite[Lemma 6.6]{MinasyanOsin2010} for a proof. 
The key fact from this lemma is that the inner automorphisms are finite index, so `most' automorphisms of a virtually cyclic group are conjugations.

\begin{lem}
\label{lem:normalisers_still_virt_cyclic}
Suppose $H$ is virtually free (of rank at least 2), and let $h$ be a non-trivial element of a finite index free subgroup $F$. Then $\langle h \rangle$ has finite index in its normaliser, which in particular is virtually cyclic.
\end{lem}

\begin{proof}
First consider the intersection $N_H(\langle h \rangle) \cap F$: this is an infinite cyclic group, generated by the root of $h$ (which we denote $\hat{h}$). This contains $\langle h \rangle$ with finite index. But $N_H(\langle h \rangle) \cap F$ itself is a finite index subgroup of $N_H(\langle h \rangle)$, which must again contain $\langle h \rangle$ with finite index.\qedhere

\end{proof}

We now combine these results to prove Proposition~\ref{prop:McCool_of_free-by-finite}.

\begin{proof}[Proof of Proposition~\ref{prop:McCool_of_free-by-finite}]
Let $F$ be a finite index normal subgroup of $H$. By Corollary~\ref{cor:OutCentraliserEnough}, it will suffice to show that the subgroup of $N$ (the isomorphic image of $\Aut_H(F)$ in $\Aut(H)$) 
that acts as conjugation on each subgroup $H_i$ is finitely generated. We use $A$ for this subgroup of $N$. 

Each subgroup $H_i$ is virtually $\Z$; in particular its intersection with $F$ is generated by a single element $h_i$. This intersection is preserved under conjugation by elements of $H_i$ (since $F$ is a normal subgroup of $H$): in particular $H_i$ is a subgroup of $N_H(\langle h_i \rangle)$. By Lemma~\ref{lem:normalisers_still_virt_cyclic}, since it contains $h_i$, it is finite index in this normaliser. 

Let $Q=\Ad(H)/\Inn(F) \cong H/FK$ be the subgroup of $\Out(F)$ induced by $H$, and denote by $\Out_Q(F)$ the centraliser of $Q$ in $\Out(F)$. This is the projection of $\Aut_H(F)$ to $\Out(F)$. By Theorem~\ref{thm:BFHcentraliser}, the subgroup $\Out_Q(F)$ preserving the conjugacy class of each $\langle h_i \rangle$ is finitely generated, so this is also true of the subgroup $A$ of $N$ 
Normalisers must be sent to normalisers, so $A$ sends $N_H(\langle h_i \rangle)$ to a conjugate of itself too. 

This normaliser is virtually cyclic, so by Lemma~\ref{lem:Out_virt_cyclic_finite} it has finitely many outer automorphisms. After composing with an inner automorphism we induce an automorphism of $N_H(\langle h_i \rangle)$, and we may restrict to those which induce an inner automorphism. This restriction gives a finite index subgroup of $A$, which acts as a conjugation on $N_H(\langle h_i \rangle)$, and in particular on the subgroup $H_i$. Repeating this for each subgroup $H_i$ (there are only finitely many) still defines a finite index subgroup, which is itself finitely generated.
\end{proof}

\begin{rem}
Notice that the ad-hoc arguments given in Proposition~\ref{prop:McCool_for_exponential} for the two cases that are not virtually cyclic can be viewed as a special case of the arguments used here for general periodic automorphisms. (Observe that $\PGL_2(\Z) \cong \Out(C_2 \ast C_2 \ast C_2)$, though the $\Out_0$ considered above would be a finite index subgroup isomorphic to $C_2 \ast C_2 \ast C_2$.) There the problem can be reduced to understanding McCool groups of free groups, allowing more complicated incident edge groups to appear while leaving the problem tractable.
\end{rem}

We are now in a position to prove one of our main theorems.

\linear*

\begin{proof}
The defining automorphism $\varphi$ has a power that is UPG and linearly growing, so it is a Dehn twist. Taking this power to define a normal subgroup $G_0$ of $G$, by Proposition~\ref{prop:canonical_tree_for_Dehn_twist}, $G_0$ has a canonical action on a tree $T$. Then by Proposition~\ref{prop:no_really_theres_an_action}, $G$ has a nearly canonical action on the same tree. The vertex stabilisers are free-by-cyclic groups which are virtually free-times-cyclic; edge stabilisers are free-by-cyclic groups that are virtually $\Z^2$ (see Proposition~\ref{prop:linear_nearly_canonical_tree}).

Analyse this action using Theorem~\ref{thm:BassJiang}. The quotient at (1) is finite since by Lemma~\ref{lem:stabs_are_free_by_cyclic} the quotient graph is. The quotient at (2) is also finitely generated. By Proposition~\ref{prop:linear_McCool_is_enough}, this will be finitely generated if and only if the relevant McCool groups are. After passing to a finite index subgroup, Proposition~\ref{prop:periodic_mccool} describes these McCool groups (up to finite index) by considering McCool groups of virtually free groups, arising by quotienting by the centre. Since the edge groups contain the centre of the vertex groups, their image under this quotient map is virtually infinite cyclic. Finally, Proposition~\ref{prop:McCool_of_free-by-finite} gives finite generation for these McCool groups, completing this part of the proof.

The edge groups are virtually $\Z^2$, and in particular virtually abelian, so by Proposition~\ref{prop:normalisers} the quotient at (3) is finitely generated. Edge and vertex groups are both (finitely generated free)-by-cyclic, so by Lemma~\ref{lem:centralisers_are_fg} the centralisers are finitely generated groups, and so is their quotient at (4). Finally, the quotient at (5) is a quotient of a finitely generated abelian group, so is itself finitely generated.

Putting this together, we see that $\Out(G)$ admits a finite index subgroup which is finitely generated, and so $\Out(G)$ itself is finitely generated, as claimed. 
\end{proof}

\section{Quadratic growth}

\subsection{Strategy}

The strategy of the proof of this section is much like the last: 
\begin{itemize}
	\item Start with a free-by-cyclic group, $G= F_3 \rtimes_{\Phi} \Z$, where $\Phi$ has quadratic growth,
	\item Consider a finite index subgroup, $G_0=F_3 \rtimes_{\Phi^r} \Z$, so that $\Phi^r$ is UPG,
	\item Find a good basis of $F_3$ for $\Phi^r$ and use this to construct a tree whose deformation space is left invariant by any automorphism of $G_0$, 
	\item Deduce that the (reduced) tree of cylinders, $T=T^*_c$, of this space is $G_0$-canonical, 
	\item Use Proposition~\ref{prop:no_really_theres_an_action} to deduce that $T$ is nearly $G$-canonical
	\item Show that $\Aut^T(G)$ is finitely generated, using Theorem~\ref{thm:BassJiang}, and conclude that $\Aut(G)$ is finitely generated.
\end{itemize}

We establish some notation. Given a group, $G$, a subgroup $H$ of $G$ and elements $g,h$ of $G$ we set: 
\begin{enumerate}[(i)]
	\item We write $g \sim h$ to denote that $g$ and $h$ are conjugate in $G$, and
	\item We write $g \sim_H h$ to denote that $g$ and $h$ are conjugate by an element of $H$ (even if $g,h$ might not themselves be elements of $H$)
\end{enumerate}

\subsection{Normal forms and a tree to act on}

First we equip ourselves with a useful representative of a UPG automorphism.

\begin{prop}
\label{prop:nice_basis}
	Suppose $\Phi$ is a UPG element of $\Out(F_3)$ of quadratic growth. Then there is a representative $\varphi \in \Phi$ and a basis $\{a,b,c\}$ of $F_3$ so that \[
\begin{array}{rcl}
& \varphi \\
a & \longrightarrow & a \\
b & \longrightarrow & b a^{-k} \\
c & \longrightarrow & h c g^{-1}, \\
\end{array}\]

where $k$ is non-zero and $h$ and $g$ are in $\langle a,b \rangle$. 
\end{prop}

This is close to \cite[Proposition 5.9]{CashenLevittBNS}; we have more control over the images of the first two generators in exchange for less control over the final generator.

\begin{proof}
By Theorem 5.1.8 of \cite{Bestvina2000}, any UPG automorphism is represented by a homotopy equivalence on a graph, $G$, such that $G$ consists of edges, $E_1, \ldots, E_k$ and the homotopy equivalence maps $E_i$ to $E_i u_{i-1}$, where the $u_{i-1}$ are closed paths involving only the edges $E_1, \ldots, E_{i-1}$ ($u_{i-1}$ may be the trivial path). 

In particular, this implies that any UPG automorphism of $F_2$ has a representative, such that with  respect to some basis, $\{ a,b\}$, the automorphism fixes $a$ and sends $b$ to $b a^{-k}$ for some $k$. 

(Briefly, if the top edge, $E_k$, were separating, then the components on removing this edge would both be homotopic to circles, and then it is easy to see that the map is homotopic to the identity relative to the initial vertex of $E_k$. If $E_k$ is not separating, then removing $E_k$ leaves a graph, homotopic to a circle, on which the map is homotopic to the identity -- giving us the $a$ -- and the $E_k$ edge becomes the $b$ basis element. Note that the layered description, which is a consequence of the UPG property, does not allow ``inversions" of these various invariant circles.)

Now, if we are given a UPG automorphism, $\Phi$, of $F_3$, the above description implies that some rank 2 free factor is left invariant, up to conjugacy -- again, remove the top edge $E_k$. Each component of the complement is invariant under the map, and there must be one of rank 2. Moreover, the restriction of $\Phi$ to this invariant free factor is also UPG -- in fact the restriction of the map has a layered form as above.

This implies that there is a basis, $\{ a,b,c\}$ of $F_3$ and a representative $\varphi \in \Phi$ such that 
 \[
\begin{array}{rcl}
	& \varphi \\
	a & \longrightarrow & a \\
	b & \longrightarrow & b a^{-k} \\
\end{array}\]

But now, since the images of $a,b,c$ must also be a basis for $F_3$, the only possibility for the image of $c$ is $h c^{\pm 1} g^{-1}$ for some $g,h \in \langle a,b \rangle$. The fact that $\Phi$ is UPG (or using the description of the map) means that the image must be $h c g^{-1}$.

Finally, note that if $k$ were to be zero, then $\Phi$ would have linear growth, hence we may conclude that $k \neq 0$.  
\end{proof}

\begin{cor}
	\label{cor:tree_for_quadratic}
	Let $G=F_3 \rtimes_{\Phi} \Z$ be a free-by-cyclic group where $\Phi$ has quadratic growth. Then $G$ has a normal finite index subgroup $G_0=F_3 \rtimes_{\Phi^r} \Z$ with presentation:
	\[
	\begin{array}{rcl}
	G_0 & = & \langle a,b,c,s \ : \ a^s = a, b^s= ba^{-k}, c^s = hcg^{-1}  \rangle \\ \\
	 & = & \langle H , c \ : \ (sh)^c = sg \rangle, \text{ where } H= \langle a,b,s \rangle 
   \end{array}
	\]
	where $g,h \in \langle a,b \rangle$ and $k \neq 0$.
	
	Moreover, $G_0$ acts on a tree, $T_0$, with one orbit of vertices, and one orbit of edges such that the vertex stabilisers are conjugates of $H= \langle a,b,s \rangle$ and edge groups are conjugates of $\langle sg \rangle = \langle sh\rangle^c $. (See Figure~\ref{fig:quadratic1}.)
	
	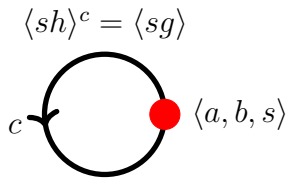
\begin{figure}
	\centering
	\usetikzlibrary{decorations.markings}
	\tikzset{middlearrow/.style n args={4}{
		decoration={
  			markings,
  				mark=at position #1 with {\arrow{#2},\node[transform shape,#4] {#3};}},postaction={decorate}},
  			middlearrow/.default={.5}{>}{}{below}	}

	\begin{tikzpicture}[scale=2]
		\draw[line width=2pt,middlearrow={.55}{>}{$c~~$}{left,rotate=75}] (0,0) arc (0:350:4mm); \node at (-0.4,0.6) {$\langle sh \rangle^c = \langle sg \rangle$};
		\draw[fill,color=red] (0,0) circle [radius=0.1]; \node at (0.5,0) {$\langle a,b,s \rangle$};
	\end{tikzpicture}
	\caption{$T_0$, as described in Corollary~\ref{cor:tree_for_quadratic}}
	\label{fig:quadratic1}
	\end{figure}
\end{cor}
\begin{proof}
Every polynomially growing automorphism has a power which is UPG, and Proposition~\ref{prop:nice_basis} provides a good generating set and the corresponding presentation. The final relations in both presentations are equivalent, realising $G_0$ as an HNN extension of $\langle a,b,s\rangle$ with stable letter $c$, and $T_0$ is the corresponding Bass-Serre tree. 	
\end{proof}

\subsection{Invariance of the tree}

\begin{prop}
	Any automorphism of $G_0$ fixes the conjugacy class of $\langle a,b,s \rangle$. That is, the deformation space defined by the tree, $T_0$ from Corollary~\ref{cor:tree_for_quadratic} is invariant under the automorphisms of $G_0$. 
\end{prop}
\begin{proof}
	Let $\psi$ denoted the automorphism of $\langle a,b,c \rangle$ induced by conjugation by $s$ and $\Psi$ the corresponding outer automorphism; as in Corollary~\ref{cor:tree_for_quadratic}, we have that: 
	
	$$
	\begin{array}{rcl}
		& \psi \\
		a & \longrightarrow & a \\
		b & \longrightarrow & b a^{-k} \\
		c & \longrightarrow & h c g^{-1}, \\
	\end{array}.
	$$
	
	Note that since $\Psi$ has quadratic growth (and does not have linear growth) it is therefore not a Dehn Twist as per Definition~\ref{DehnandUPG}. 
	Hence, by Theorem~\ref{definetwists}, the inequality in the Bestvina-Handel Theorem, Theorem~\ref{Bestvina-Handel}, is strict for $\Psi$. (One can make the same deduction from the inequalities provided by Proposition 4.2 of \cite{LevittInequalities}.)

	To summarise, for our $\Psi$ we have that: 
	$$
	\sum \max\{ \rank(\fp{\psi_i}) - 1, 0 \} \leq 1 < 2 = \rank(F_3) -1, 
	$$
	where the sum is taken over the isogredience classes in $\Psi$.
	
	Moreover, if we take $\psi_0=\psi$, then we know that $\psi$ fixes $\langle a, bab^{-1} \rangle$ and so
	
		$$
	\sum \max\{ \rank(\fp{\psi_i}) - 1, 0 \} = 1,
	$$
	
	and therefore this sum has exactly one non-zero term. In particular, this means that any automorphism, $\psi \Ad(w)$ which has a fixed subgroup that is not cyclic, is isogredient to $\psi$. The same conclusion holds if we replace $\psi$ (and $\Psi$) with $\psi^{\ell}$ (and $\Psi^{\ell}$) for some $0 \neq \ell \in \Z$. 
	
	Now let $\chi$ be an automorphism of $G_0$. First we will see that $\chi(s)$ is a conjugate of $s^{\pm1}$, and then we prove that in the case it is fixed or inverted, the subgroup $\langle a,b,s \rangle$ is preserved. 
	
	We start by observing that $\chi(s) \not\in \langle a,b,c \rangle$. This is because the subgroup $\langle a ,bab^{-1}, s \rangle \cong F_2 \times \Z$ is the centraliser of $s$ and so $\chi(\langle a ,bab^{-1}, s \rangle )$ is the centraliser of $\chi(s)$. But the centraliser of any element of $\langle a ,b, c \rangle$ is either cyclic or virtually $\Z^2$, so cannot contain a free subgroup of rank 2. Hence $\chi(s)=s^\ell w$ for some $0 \neq \ell$ and $w \in \langle a,b, c \rangle$.

	 Since $\chi(\langle a ,bab^{-1}, s \rangle ) \cong F_2 \times \Z$, we cannot have that $\chi(\langle a ,bab^{-1}, s \rangle )\cap \langle a,b,c \rangle$ is cyclic (or trivial).  However, $\chi(\langle a ,bab^{-1}, s \rangle ) $ is the centraliser of $s^{\ell}w$ and so $\chi(\langle a ,bab^{-1}, s \rangle ) \cap \langle a,b,c \rangle$ is exactly the fixed subgroup of $\psi^{\ell} \Ad(w)$. Hence, by the discussion above,  $\psi^{\ell} \Ad(w)$ and $\psi^{\ell}$ are isogredient. This implies that $s^\ell w$ and $s^{\ell}$ are conjugate in $G_0$ (since the centraliser of $\langle a,b,c \rangle$ in $G_0$ is trivial). Thus $\chi(s)$ is conjugate to $s^{\ell}$, which implies that $\ell = \pm 1$, as $s$ has no roots. Thus we conclude that $\chi(s)$ is conjugate to $s^{\pm 1}$ and so $\langle a ,bab^{-1}, s \rangle$ is sent to a conjugate by $\chi$.

	Therefore, up to composing $\chi$ with an inner automorphism of $G_0$, we may assume that $s$ is fixed or inverted by $\chi$, and we consider the images of $a$ and $b$. We write $u$ for the image of $a$ and $s^n v$ for the image of $b$, where $u,v$ are elements of $\langle a,b,c \rangle$. (Notice that the relation $b^s=ba^{-k}$ implies that the image of $a$ lies in this free group.)
	
	The image of $bab^{-1}$ is $vuv^{-1}$, so $\langle u,vuv^{-1} \rangle = \langle a,bab^{-1} \rangle$. In particular, this shows that $u$ and $v$ (by malnormality of free factors) are contained in $\langle a,b \rangle$.
	
	To see the other inclusion, notice that we also know that $\langle u,v \rangle$ contains $\langle a, bab^{-1} \rangle$. By considering the Stallings graphs (see~\cite{Stallings1983}) of both subgroups, this means it contains either $\langle a, bab^{-1} \rangle$ or $\langle a,b \rangle$ as a free factor.
	
	That is, the subgroup inclusion gives a graph morphism from the Stallings graph of $\langle a, bab^{-1} \rangle$ to that of $\langle u,v \rangle$ with respect to the basis $\{ a,b\}$. If this map is injective, then the Stallings graph of $\langle a, bab^{-1} \rangle$ is a subgraph and therefore $\langle a, bab^{-1} \rangle$ is a free factor of $\langle u,v \rangle$. If not, then the two vertices of the Stallings graph of  $\langle a, bab^{-1} \rangle$ are identified, and we must get that $\langle a,b \rangle$ is a free factor of, and hence must be equal to, $\langle u,v \rangle$. This is an easy version of the arguments in \cite{Ventura1997}, Theorem 1.7.
	
	 Since it has rank 2, this actually says $\langle u,v \rangle$  is equal to either $\langle a, bab^{-1} \rangle$ or $\langle a,b \rangle$; the first is impossible since it would imply that  $\langle u,v \rangle =\langle a, bab^{-1} \rangle =   \langle u,vuv^{-1} \rangle $, which cannot happen since the last subgroup does not contain $v$. 
	 
	 Hence, $\langle u,v \rangle = \langle a,b \rangle$ and  $\langle u,v,s \rangle = \langle a,b,s \rangle$. \end{proof}

\begin{cor}
	The (collapsed) tree of cylinders, $T^*_c$, of $T_0$ is $G_0$-canonical and hence nearly $G$-canonical.
\end{cor}
\begin{proof}
The fact that the deformation space of $T_0$ is invariant, gives us that the (collapsed) tree of cylinders, $T^*_c$ is canonical, see Subsection~\ref{treesofcylinders}. 

Then Proposition~\ref{prop:no_really_theres_an_action} gives us the second statement. 
\end{proof}

\subsection{Calculating the tree of cylinders, \texorpdfstring{$T_c$}{Tc}}

Our goal now is to calculate $T_c$. In order to do this, we actually modify the basis given by Proposition~\ref{prop:nice_basis}. The tree, $T_0$ from Proposition~\ref{cor:tree_for_quadratic} remains the same, but these modifications aid the calculation. 

Throughout this subsection, we are working with the subgroup 
$$G_0 =  \langle a,b,c,s \ : \ a^s = a, b^s= ba^{-k}, c^s = hcg^{-1}  \rangle.$$

First we observe that we can modify the elements $g,h$ from Proposition~\ref{prop:nice_basis}, and thus in the description of $G_0$. 

\begin{lem}
	\label{lem:h_and_g_are_not_unique}
	The choices of $g$ and $h$ in the statement of Proposition~\ref{prop:nice_basis} are not unique. In particular, if $sh \sim_{\langle a, b \rangle} sh'$ and $sg \sim_{\langle a, b \rangle} sg'$, then there exist $x,y \in \langle a,b \rangle$ such that if $c'=x^{-1} c y$, then the image of $c'$ under $\varphi$ is $h' c' g'^{-1}$.  
	
\end{lem}
\begin{proof}
	We will work in the corresponding free-by-cyclic group, $G_0$ from Proposition~\ref{cor:tree_for_quadratic} and its presentation. 
	
	Recall that  $G_0  =  \langle a,b,c,s \ : \ a^s = a, b^s= ba^{-k}, c^s = hcg^{-1}  \rangle$. It will be sufficient to show that $ s^{-1} {c'} s =    h' c' g'^{-1}$.
	
	Suppose $(sh)^x=sh'$, and $(sg)^y=sg'$, where $x$ and $y$ are elements of $\langle a,b \rangle$. Then put $c' = x^{-1} c y$. 
	
	We get that,
	{%
	\addtolength{\jot}{2pt} 
	\begin{align*}
		s^{-1} {c'} s &=  (x^{-1} c y)^s \\
		& =   x^{-s} hc g^{-1} y^s  \\
		&  =    (x^{-s} h x)  c'(y^{-1} g^{-1} y^s)  \\
		& =  s^{-1} (sh)^x  c' (sg)^{-y} s\\
		&  =  h' c' g'^{-1}. \qedhere
	\end{align*}
	}
\end{proof}

Note that each of $sh$ and $sg$ normalise $\langle a,b \rangle$. Moreover, they induce the same outer automorphism, and this is a Dehn Twist of $\langle a,b \rangle$. However, while $sh$ and $sg$ are conjugate in $G_0$ -- and so induce isogredient automorphisms of $\langle a,b,c \rangle$ -- they might not induce isogredient automorphisms of $\langle a,b \rangle$. 

One key point is that:

\begin{lem} The following are equivalent: 
	\begin{enumerate}[(i)]
		\item  $sh$ and $sg$ induce isogredient automorphisms on $\langle a, b \rangle$
		\item $sh \sim_{\langle a, b \rangle} sg$ 
		\item $sh \sim_{\langle a,b,s \rangle} sg$.
	\end{enumerate}
\end{lem}
\begin{proof}
	The first two are clearly equivalent, and notice that $\langle a,b,sh \rangle = \langle a,b,s \rangle = \langle a,b,sg \rangle$, which makes the second and third equivalent since we can choose the new generator so it centralises the conjugated element. 
\end{proof}

We will use the following result, to help us modify $g$ and $h$ as above. 

\begin{cor}[{\cite[Corollary 3.10]{DehnTwists}}]
	\label{cor:DehnTwists3_10}
	Let $\Psi \in \Out(F_n)$, $n\geq 2$, be a Dehn Twist outer automorphism fixing a conjugacy class. Then there is a $\psi \in \Psi$ with fixed subgroup of rank at least two fixing an element of that conjugacy class.
\end{cor}

\begin{lem}
	In $G_0$, the centraliser $C_{\langle a,b \rangle}(sh)$ has rank 0,1 or 2. If the rank is at least 1, then $sh \sim_{\langle a, b \rangle} sh'$ for some $h' \in C_{\langle a,b \rangle}(s) = \langle a, bab^{-1} \rangle$. The same is true for $g$. 
	
	Moreover, one of 	$C_{\langle a,b \rangle}(sh)$ and 	$C_{\langle a,b \rangle}(sg)$ has rank 0 (is the trivial group). 
\end{lem}
\begin{proof}
The first statement follows from the Bestvina-Handel Theorem, Theorem~\ref{Bestvina-Handel}.

For the second statement, we invoke Corollary~\ref{cor:DehnTwists3_10}, to say that if $C_{\langle a,b \rangle}(sh)$ is non-trivial, then there exists a non-trivial $w \in \langle a,b \rangle$ and an $x \in \langle a,b \rangle$ such that:
\[
\begin{array}{rcl}
(w^x)^{sh}  & =  & w^x \\
w^{s} & = & w.
\end{array}
\]
Here we are using Theorem~\ref{Bestvina-Handel} to say that since the underlying free group has rank 2, there is exactly one isogredience class with fixed subgroup of rank at least 2, and hence the  $\psi$ from Corollary~\ref{cor:DehnTwists3_10} is, without loss of generality, the automorphism induced by conjugation by $s$ (on $\langle a,b \rangle$). (It is more convenient for the following argument to write $w^x$ for the element fixed by the automorphism induced by $sh$.)

But these equations imply that, 
\[
w^{s^{-1} x (sh) x^{-1}} = w^{ x (sh) x^{-1}} = (w^x)^{(sh) x^{-1}} = (w^x)^{x^{-1}} = w. 
\]

Hence, as $w$ is non-trivial and both $w$ and $s^{-1} x (sh) x^{-1}$ are elements of $\langle a,b \rangle$, we get that $s^{-1} x (sh) x^{-1} \in \langle w \rangle \leq C_{\langle a,b \rangle}(s)$, and hence $sh\sim_{\langle a, b \rangle} sw^m$, for some $m \in \Z$ (without loss of generality, we can assume $w$ is not a proper power, and so generates its own centraliser in $\langle a,b \rangle$). The same calculation gives the result for $g$. 

Finally, notice that if both $h, g \in C_{\langle a,b}(s)$, then $\Phi$ has linear growth. Thus, via Lemma~\ref{lem:h_and_g_are_not_unique}, we deduce that one of $C_{\langle a,b \rangle}(sh)$ and 	$C_{\langle a,b \rangle}(sg)$ has rank 0.
\end{proof}

\begin{rem}
	\label{rem:trivial_centraliser}
	Given this result, we shall henceforth assume that $C_{\langle a,b \rangle}(sg)$ is the trivial group. (Note that $h$ and $g$ can be interchanged by replacing $c$ with $c^{-1}$ so there is no loss of generality in assuming this.)
\end{rem}

We also record that, 

\begin{lem}
\label{lem:normaliser_is_centraliser}
Let $G$ be a free-by-cyclic group, with stable letter $s$. Any subgroup $\langle s^m w \rangle$, with $m \neq 0$, has centraliser and normaliser equal.
\end{lem}

\begin{proof}
Notice that conjugation by any element of the normaliser induces an automorphism of $\langle s^m w \rangle$, and in particular either preserves the generator (in which case it is an element of the centraliser) or inverts it. However, conjugating cannot affect the exponent sum of the stable letter $s$, and so this last case does not arise. 
\end{proof}

Since there is only one orbit of edges, we can understand the cylinders by understanding the normaliser of any edge stabiliser. Since the edges are stabilised by infinite cyclic groups of the kind discussed in Lemma~\ref{lem:normaliser_is_centraliser}, this is equivalent to understanding their centralisers.

\begin{thm}
	\label{thm:tree_for_quadratic}
	Let $G_0  =  \langle a,b,c,s \ : \ a^s = a, b^s= ba^{-k}, c^s = hcg^{-1}  \rangle$, and $T_0$ be the Bass-Serre tree on which $G_0$ acts via the HNN decomposition,  $G_0 = \langle H , c \ : \ (sh)^c = sg \rangle$, where $H= \langle a,b,s \rangle$. 
	
	Moreover, assume that $C_{\langle a,b \rangle}(sg)$ is the trivial group, as in Remark~\ref{rem:trivial_centraliser}.

	We form the tree of cylinders, $T_c$, and collapsed tree of cylinders $T_c^*$ taking maximal infinite cyclic groups to be the family $\mathcal{E}$ and equality to be the admissible equivalence relation.
	
	\begin{itemize}
		\item If $sh \not\sim_{\langle a, b \rangle} sg$, then $T_c^*=T_0$, or $T_c=T_c^*$ is simply a subdivision of $T_0$.
		\item If $sh \sim_{\langle a, b \rangle} sg$, then $T_c=T_c^*$ has one edge orbit, with infinite cyclic stabilisers, conjugates of $\langle sh \rangle$, and two vertex orbits, with stabilisers conjugates of $\langle a,b,s \rangle$ and $C(sh) \cong \Z^2$. 
	\end{itemize}
	
\end{thm}

\begin{proof}
Since $T_0$ has one orbit of edges and one orbit of vertices, the tree of cylinders of $T_0$ will have two orbits of vertices -- one for the cylinders, and one for the $T_0$-vertices. 

Since our relation is equality, edge stabilisers in $T_0$ are conjugate to $sh$, and we have Lemma~\ref{lem:normaliser_is_centraliser}, we deduce that a cylinder is the orbit of an edge under the action of the centraliser of the edge stabiliser (in $G_0$). 

As $G_0$ acts without inversions on $T_0$, we may equivariantly orient the edges of $T_0$. A vertex stabiliser in $T_0$ acts on the incident edges with two orbits -- one orbit for the incoming edges, and one for the outgoing edges. 
	
Choose this orientation so that at the vertex stabilised by $\langle a,b,s \rangle$, the incoming edges have stabiliser conjugate (in $\langle a,b,s \rangle$) to $\langle sg \rangle$ and for the outgoing edges it is conjugate to  $\langle sh \rangle$.

The fact that $C_{\langle a,b \rangle}(sg)$ is the trivial group implies that  $C_{\langle a,b,s \rangle}(sg) = \langle sg \rangle$ and hence that no cylinder may contain two incoming edges at a vertex.

\textbf{\underline{Suppose $sh \not\sim_{\langle a, b \rangle} sg$: }}

If a cylinder contained both incoming and outgoing edges at a vertex, then (moving back to the vertex stabilised by $\langle a,b,s \rangle$) we would have $sh \sim_{\langle a,b,s \rangle} sg$, since acting on the edges conjugates the stabilisers. So if $sh \not\sim_{\langle a, b \rangle} sg$ (which is equivalent to $sh \not\sim_{\langle a,b,s \rangle} sg$), then no cylinder may contain both incoming and outgoing edges at a vertex. 

Thus if $sh \not\sim_{\langle a, b \rangle} sg$, all cylinders consist of a collection of outgoing edges from a vertex. More concretely, if we take the edge with stabiliser $\langle sh \rangle$, then the corresponding cylinder consists of edges starting from the vertex with stabiliser  $\langle a,b,s \rangle$, and are thus all in the same $\langle a,b,s \rangle$-orbit. In particular, this implies that $C(sh) = C_{\langle a,b,s \rangle}(sh) = C_{\langle a,b \rangle}(sh) \times \langle sh \rangle$. 

The cylinder stabiliser acts with two orbits on its vertices -- the central vertex and all the rest, and hence the tree of cylinders of $T_0$ has two edge orbits corresponding to these different inclusions. One of these edges has stabilisers equal to the edge stabilisers of the original tree (this is where we have the vertex being one of the `outside' vertices of the cylinder), whereas the other edge group is equal to the stabiliser of the cylinder, (conjugates of) $C(sh)$. 

If $C(sh)$ is not cyclic, then the collapsed tree of cylinders will collapse the corresponding edge, and we will return to the original tree. 

If $C(sh)$ is cyclic, then the tree of cylinders is just a subdivision of $T_0$ -- we have subdivided an edge, and given the new vertex the same stabiliser as the edge it is part of.

\textbf{\underline{Suppose $sh \sim_{\langle a, b \rangle} sg$: }}

If $sh \sim_{\langle a, b \rangle} sg$, then we orient the edges of $T_0$ as before and now we get that both $C_{\langle a,b \rangle}(sh)$  and $C_{\langle a,b \rangle}(sg)$ are trivial (since they are conjugate). Therefore, $C_{\langle a,b, s\rangle}(sh) = \langle sh \rangle$ and  $C_{\langle a,b, s\rangle}(sg) = \langle sg \rangle$.  

This means that a cylinder cannot contain either two outgoing or two incoming edges at any vertex. However, each cylinder does contain both an outgoing and incoming edge at each vertex. Hence the cylinder is a line and it is straightforward to verify that $C(sh) \cong \Z^2$. (Since $sh \sim_{\langle a, b \rangle} sg$, we may assume that $h=g$, and in this case, $C(sh) = \langle c, sh \rangle$ -- $c$ is acting as a translation, and therefore transitively on the vertices and edges of this line). 

In this case, there are again two orbits of vertices in the tree of cylinders -- one for the cylinders, one for the vertices of $T_0$ -- with stabilisers (conjugates of) $\langle a,b,s \rangle$ and $C(sh) \cong \Z^2$.  

Since the cylinder stabiliser acts transitively on its vertices, there is only one edge, whose stabiliser is (the conjugates of) $\langle sh \rangle$. \qedhere

\end{proof}

\begin{rem}
The tree of cylinders produced by this theorem realises a maximal preserved free factor system for the automorphism induced by $s$: it is an interesting question if this is true more generally (say, in higher rank or higher polynomial growth).
\end{rem}~

We now use Theorem~\ref{thm:tree_for_quadratic} to provide a nearly canonical tree for the general (not just UPG) case.

\begin{cor}
	Let $G \cong F_3 \rtimes_\Phi \Z$, and $\Phi$ is quadratically growing. Then $G$ admits an action on a nearly canonical tree, $T$, such that:
	\begin{enumerate}[(i)]
		\item The action is co-compact (equivalently, co-finite),
		\item Edge stabilisers are infinite cyclic,
		\item Vertex stabilisers are of the form $F_r \rtimes \Z$, where $r=0,1,2$. 
	\end{enumerate} 
\end{cor}
\begin{proof}
We simply apply Proposition~\ref{prop:no_really_theres_an_action} to the collapsed tree of cylinders for $G_0$ above, Theorem~\ref{thm:tree_for_quadratic}, to get a nearly canonical action on the same tree. The fact that the $G$ action extends the $G_0$ action tells us about the stabilisers. (For example, edge stabilisers in $G$ must be infinite cyclic since their intersection with $F_3$ is trivial). 
\end{proof}

We now use this to prove the following theorem, which is part of Theorem~\ref{thm:rank3}.

\begin{thm}
\label{thm:quadratic}
Suppose $G \cong F_3 \rtimes_\varphi \Z$, and $\varphi$ is quadratically growing. Then $\Out(G)$ is finitely generated.
\end{thm}

\begin{proof}
We use the tree constructed above, and we calculate the quotients of $\Out^T(G)$ described in Theorem~\ref{thm:BassJiang}. The quotient graphs are finite, and therefore so is the quotient at (1). For the quotient at (2), the edge groups are all infinite cyclic, and therefore have finite outer automorphism group. So by Lemma~\ref{lem:quotient_by_McCool}, we only need to check the McCool groups of vertex groups. Since vertex groups are free by cyclic groups of rank 0, 1 or 2, these are finitely generated by Proposition~\ref{prop:McCool_for_exponential}. 

Since the edge groups are infinite cyclic, we may apply Proposition~\ref{prop:normalisers} to see that the quotient at (3) is finite. The quotient at (4) is finitely generated by Lemma~\ref{lem:centralisers_are_fg} and that at (5) as a quotient of a finitely generated abelian group.
\end{proof}

Our other main theorem is proved by combining Theorem~\ref{thm:linear} (restricted to rank 3) for the linear growth case, Theorem~\ref{thm:quadratic} for the quadratic growth case, Theorem~\ref{thm:exponential} for the exponential case and \cite{LevittGBS} for the periodic case.

\rankthree*

\bibliographystyle{plain}

\end{document}